\documentclass[12pt]{amsart}

\usepackage{amsmath}
\usepackage{amsfonts}
\usepackage{latexsym}
\usepackage{graphicx}
\usepackage{amssymb}
\usepackage{amsthm}
\usepackage[margin=2cm]{geometry}
\usepackage{url}
\usepackage{color}
\usepackage{enumerate}
\usepackage[shortlabels]{enumitem} %pour commencer les enumerations a des nombres differents
\usepackage[small]{caption}
\usepackage{cite}       % sort and compress references number in latex

\usepackage{longtable}

\usepackage{hyperref}
\usepackage{cleveref}

\usepackage{todonotes}

\usepackage{array}           % for \newcolumntype macro
\usepackage{stmaryrd}

\usepackage{leftidx} % correct spacing for exponent on the left \leftidx{^\infty}

\usepackage[T1]{fontenc}     % Hyphénation des mots accentués
\usepackage{lmodern}         % Polices vectorielles
\usepackage[utf8]{inputenc}  % Codage UNICODE (UTF-8)

\newtheorem{theorem}{Theorem}[section]
\newtheorem{lemma}[theorem]{Lemma}
\newtheorem{proposition}[theorem]{Proposition}

\newtheorem{question}[theorem]{Question}
\newtheorem{corollary}[theorem]{Corollary}
\newtheorem{conjecture}[theorem]{Conjecture}
\newtheorem{definition}[theorem]{Definition}

\newtheorem*{frobconjecture*}{Markoff Injectivity Conjecture}

\newtheorem{maintheorem}{Theorem}
\newtheorem{maincorollary}{Corollary}

\renewcommand\a{\symb{0}}
\renewcommand\b{\symb{1}}

\newcommand{\Z}{\mathbb{Z}}
\newcommand{\Q}{\mathbb{Q}}
\newcommand{\R}{\mathbb{R}}

\newcommand{\Bcal}{\mathcal{B}}

\newcommand{\Lcal}{\mathcal{L}}

\newcommand{\GL}{\mathrm{GL}}
\newcommand{\SL}{\mathrm{SL}}
\newcommand{\occ}{\mathrm{occ}}
\newcommand{\symb}[1]{\mathtt{#1}}		% Symbol

\begin{document}

\title
[$q$-analog of the Markoff injectivity conjecture]
{The $q$-analog of the Markoff injectivity conjecture\\ over the language of a balanced sequence}

% shorter title ?
%\title{A $q$-analog of the Markoff injectivity conjecture}

% Adresse LaBRI selon
% https://www.labri.fr/index.php?n=LaBRI.HowToSigne

\author[S.~Labb\'e]{S\'ebastien Labb\'e}
\address[S.~Labb\'e]{Univ. Bordeaux, CNRS, Bordeaux INP, LaBRI, UMR 5800, F-33400, Talence, France}
\email{sebastien.labbe@labri.fr}
\urladdr{http://www.slabbe.org/}

\author[M.~Lapointe]{Mélodie Lapointe}
\thanks{This work was supported by the Agence Nationale de la Recherche through the project Codys (ANR-18-CE40-0007). The second author acknowledges the support of the Natural Sciences and Engineering Research Council of Canada (NSERC), [funding reference number BP–545242–2020] and the support of the Fonds de Recherche du Québec en Science et Technologies.}
\address[M.~Lapointe]{Université de Paris, CNRS, IRIF, F-75006 Paris, France}
\email{lapointe@irif.fr}
\urladdr{http://lapointemelodie.github.io/}

\keywords{Balance \and Markoff spectrum \and Sturmian \and Christoffel \and $q$-analog}
\subjclass[2010]{Primary 11J06; Secondary 68R15 \and 05A30}

% 11J06: Markov and Lagrange spectra and generalizations
% 68R15: Combinatorics on words
% 05A30: $q$-calculus and related topics
% 37B10: Symbolic dynamics

% Reutenauer's book:
% 11J06: Markov and Lagrange spectra and generalizations
% 68R15: Combinatorics on words
% 
% Aigner book:
% 20E05: Free nonabelian groups
% 20H10: Fuchsian groups and their generalizations (group-theoretic aspects
%
% Saliola's book
% 37B10: Symbolic dynamics
%
% Morier
% MRCLASS = {11A55 (05A30 11B57 13F60 57K14)},
% 11A55: Continued fractions
% 05A30: $q$-calculus and related topics
% 11B57: Farey sequences; the sequences $1^k, 2^k, \dots$
% 13F60: Cluster algebras

\date{\today}

\begin{abstract}
The Markoff injectivity conjecture states that $w\mapsto\mu(w)_{12}$ is injective on the set of Christoffel words where $\mu:\{\mathtt{0},\mathtt{1}\}^*\to\mathrm{SL}_2(\mathbb{Z})$ is a certain homomorphism and $M_{12}$ is the entry above the diagonal of a $2\times2$ matrix $M$. Recently, Leclere and Morier-Genoud (2021) proposed a $q$-analog $\mu_q$ of $\mu$ such that $\mu_{q}(w)_{12}|_{q=1}=\mu(w)_{12}$ is the Markoff number associated to the Christoffel word $w$ when evaluated at $q=1$. We show that there exists an order $<_{radix}$ on $\{\mathtt{0},\mathtt{1}\}^*$ such that for every balanced sequence $s \in \{\mathtt{0},\mathtt{1}\}^\mathbb{Z}$ and for all factors $u, v$ in the language of $s$ with $u <_{radix} v$, the difference $\mu_q(v)_{12} - \mu_q(u)_{12}$ is a nonzero polynomial of indeterminate $q$ with nonnegative integer coefficients. Therefore, the map $u\mapsto\mu_q(u)_{12}$ is injective over the language of a balanced sequence. The proof uses an equivalence between balanced sequences satisfying some Markoff property and indistinguishable asymptotic pairs. 
\end{abstract}

\maketitle

\section{Introduction}

A Markoff triple is a positive solution of the Diophantine equation
\begin{align}~\label{Eq:Markoff}
    x^2 + y^2 + z^2 = 3xyz.
\end{align}
It was introduced by Markoff\cite{M1879,MR1510073} to describe minima of indefinite real binary quadratic forms.
Positive solutions of Equation~\ref{Eq:Markoff} can be computed recursively.
If $(x,y,z)$ is a Markoff triple, then $(x,3xy-z,y)$ and $(y,3yz-x,z)$ are also Markoff triples, see Figure~\ref{fig:tree_rule}.
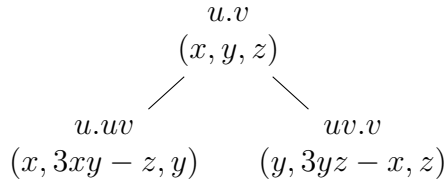
\begin{figure}[h]
\begin{tikzpicture}
\tikzstyle{level 1}=[sibling distance=8em]
    \node[align=center] {$u.v$ \\ $(x,y,z)$ }
        child {node[align=center]{$u.uv$ \\ $(x, 3xy-z,y)$}}
        child {node[align=center]{$uv.v$ \\ $(y, 3yz-x,z)$}};
\end{tikzpicture}
    \caption{Binary tree structure of Christoffel words $u.v$ and Markoff triples $(x,y,z)$.}
    \label{fig:tree_rule}
\end{figure}
A Markoff triple is called proper as long as $x$, $y$ and $z$ are pairwise distinct. 
There are only two Markoff triples which are not proper, namely $(1,1,1)$ and $(1,2,1)$. 
If $(x,y,z)$ is a proper Markoff triple with $y \geq x$ and $y \geq z$, then $(x, 3xy-z,y) \neq (y, 3yz-x,z)$ and both $3xy -z$ and $3yz-x$ are greater than $y$.
Hence, the proper Markoff triples naturally label a complete infinite binary tree, see Figure~\ref{fig:tree}.
All proper Markoff triple have a maximum value. 
Frobenius \cite{F1913} asked whether each Markoff number (an element of a
Markoff triple) is the maximum of a unique Markoff triple. 
The question known as the \emph{uniqueness conjecture} is still open. 
A book was devoted to it and its many equivalent formulations
for its $100^{th}$ anniversary \cite{MR3098784}.

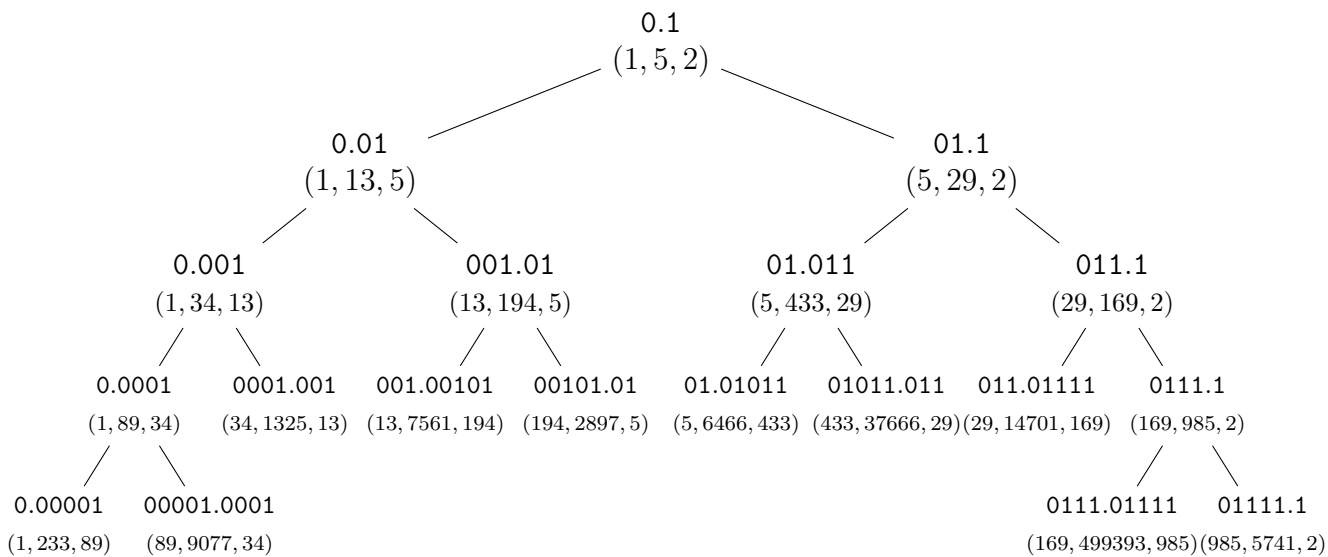
\begin{figure}[h]
\begin{center}
    \begin{tikzpicture}[yscale=.8]
        %\begin{scope}[xshift=-5.5cm,yshift=11cm]
        %\tikzstyle{every node}=[font=\footnotesize]
        %\tikzstyle{level 1}=[sibling distance=8em]
        %    \node[align=center] {$u.v$ \\ $(x,y,z)$ }
        %        child {node[align=center]{$u.uv$ \\ $(x, 3xz-y,z)$}}
        %        child {node[align=center]{$uv.v$ \\ $(z, 3zy-x,y)$}};
        %    \draw (-2.8,-2) rectangle (2.8,.8);
        %\end{scope}
    \node[align=center] (A0) at (0,10) {$\a.\b$\\   $(1,5,2)$};
    \node[align=center] (B0) at (-4,8) {$\a.\a\b$\\  $(1,13,5)$};
    \node[align=center] (B1) at ( 4,8) {$\a\b.\b$\\  $(5,29,2)$};
    %\tikzstyle{every node}=[font=\footnotesize]
    \node[align=center] (C0) at (-6,6) {$\a.\a\a\b$\\  \footnotesize$(1,34,13)$};
    \node[align=center] (C1) at (-2,6) {$\a\a\b.\a\b$\\ \footnotesize$(13,194,5)$};
    \node[align=center] (C2) at ( 2,6) {$\a\b.\a\b\b$\\ \footnotesize$(5,433,29)$};
    \node[align=center] (C3) at ( 6,6) {$\a\b\b.\b$\\  \footnotesize$(29,169,2)$};
    %\tikzstyle{every node}=[font=\tiny]
    \node[align=center] (D0) at (-7,4) {\footnotesize$\a.\a\a\a\b$\\    \tiny$(1,89,34)$};
    \node[align=center] (D1) at (-5,4) {\footnotesize$\a\a\a\b.\a\a\b$\\  \tiny$(34,1325,13)$};
    \node[align=center] (D2) at (-3,4) {\footnotesize$\a\a\b.\a\a\b\a\b$\\ \tiny$(13,7561,194)$};
    \node[align=center] (D3) at (-1,4) {\footnotesize$\a\a\b\a\b.\a\b$\\  \tiny$(194,2897,5)$};
    \node[align=center] (D4) at ( 1,4) {\footnotesize$\a\b.\a\b\a\b\b$\\  \tiny$(5,6466,433)$};
    \node[align=center] (D5) at ( 3,4) {\footnotesize$\a\b\a\b\b.\a\b\b$\\ \tiny$(433,37666,29)$};
    \node[align=center] (D6) at ( 5,4) {\footnotesize$\a\b\b.\a\b\b\b\b$\\ \tiny$(29,14701,169)$};
    \node[align=center] (D7) at ( 7,4) {\footnotesize$\a\b\b\b.\b$\\    \tiny$(169,985,2)$};
    \node[align=center] (E0) at (-8,2) {\footnotesize$\a.\a\a\a\a\b$\\   \tiny$(1,233,89)$};
    \node[align=center] (E1) at (-6,2) {\footnotesize$\a\a\a\a\b.\a\a\a\b$\\\tiny$(89,9077,34)$};
    \node[align=center] (E6) at ( 6,2) {\footnotesize$\a\b\b\b.\a\b\b\b\b$\\\tiny$(169,499393,985)$};
    \node[align=center] (E7) at ( 8,2) {\footnotesize$\a\b\b\b\b.\b$\\   \tiny$(985,5741,2)$};
    \draw (A0) -- (B0);
    \draw (A0) -- (B1);
    \draw (B0) -- (C0);
    \draw (B0) -- (C1);
    \draw (B1) -- (C2);
    \draw (B1) -- (C3);
    \draw (C0) -- (D0);
    \draw (C0) -- (D1);
    \draw (C1) -- (D2);
    \draw (C1) -- (D3);
    \draw (C2) -- (D4);
    \draw (C2) -- (D5);
    \draw (C3) -- (D6);
    \draw (C3) -- (D7);
    \draw (D0) -- (E0);
    \draw (D0) -- (E1);
    \draw (D7) -- (E6);
    \draw (D7) -- (E7);
\end{tikzpicture}
\end{center}
    \caption{Binary tree of proper Christoffel words and proper Markoff triples}
    \label{fig:tree}
\end{figure}

The conjecture can be stated in terms of Christoffel words \cite{MR3887697}.
Christoffel words are words over the alphabet $\{\a,\b\}$ also satisfying
a binary tree structure: $\a$, $\b$ and $\a\b$ are Christoffel words and if
$u,v,uv\in\{\a,\b\}^*$ are Christoffel words then $uuv$ and $uvv$ are Christoffel
words \cite{MR2464862},
see Figure~\ref{fig:tree_rule}.
Note that these are usually named \emph{lower} Christoffel words.
It is known that each Markoff number can be expressed in terms of a Christoffel
word.
More precisely, 
let $\mu$ be the monoid homomorphism $\{\a,\b\}^* \rightarrow \SL_2(\mathbb{Z})$ defined by 
\[
    \mu(\a) =
\left(\begin{array}{cc}
    2 & 1 \\
    1 & 1
\end{array}\right)
\quad
\text{ and }
\quad
    \mu(\b) = 
\left(\begin{array}{cc}
    5 & 2 \\
    2 & 1
\end{array}\right).
\] 
Each Markoff number is equal to $\mu(w)_{12}$ for some Christoffel word $w$ 
(see \cite{R2009}), that is, the element above the diagonal in the matrix $\mu(w)$.
Moreover, positive primitive elements of the free group
$F_2$ on two generators are in bijection with Markoff triples~\cite{Bombieri2007}. 
Both results are equivalent since primitive elements of the free group $F_2$
are in bijection with Christoffel words \cite{MR2295123}.

In other words, the map $w\mapsto\mu(w)_{12}$ from Christoffel words to Markoff
numbers is surjective.
For example, the Markoff number 194 is associated with the Christoffel word $\a\a\b\a\b$
as it is the entry at position $(1,2)$ in the matrix
$\mu(\a\a\b\a\b)=
\left(\begin{smallmatrix}
    463 & 194 \\
    284 & 119
\end{smallmatrix}\right)$.
Whether this map provides a bijection between Christoffel
words and Markoff numbers is equivalent to the uniqueness conjecture.
Indeed, the uniqueness conjecture can be expressed in terms of the
injectivity of the map $w\mapsto\mu(w)_{12}$ \cite[\S 3.3]{MR3887697}.

% Many combinatorial objects like Farey numbers, Cohn matrices and Christoffel
% words also label a complete infinite binary tree. 

\begin{frobconjecture*}
    The map $w\mapsto\mu(w)_{12}$ is injective on the set of Christoffel words.
\end{frobconjecture*}

The map $w\mapsto\mu(w)_{12}$ is defined over the monoid $\{\a,\b\}^*$ not
only on Christoffel words.
On this extended domain, Lapointe and Reutenauer showed that
$w\mapsto\mu(w)_{12}$ is strictly increasing (thus injective) over the language of
factors appearing in a Christoffel word 
\cite{lapointe_these_2020,LR2021}, thus also for all Christoffel words on an infinite path in the binary tree of Christoffel words.
The map is not injective on $\{\a,\b\}^*$ as for example,
$\mu(\a\a\b\b)_{12} = 75 = \mu(\a\b\a\b)_{12}$.
But Lapointe and Reutenauer conjectured that it is injective on the language of
all factors of Christoffel words \cite[Conjecture 2]{LR2021}.

\subsection{$q$-analogs}

Markoff numbers and the
Markoff injectivity conjecture can be parametrized by introducing a parameter~$q$.
Recall that the $q$-analog of a nonnegative integer $n$ is
\[
    [n]_q = 1+q+\dots+q^{n-1}=\frac{1-q^n}{1-q}.
\]
Recently, 
the $q$-analog 
    $\left[\frac{a}{b}\right]_q\in \Q(q)$
of every rational number $\frac{a}{b}\in\Q$ was introduced
to be a ratio of polynomials over $q$
defined from the continued fraction expansion of $\frac{a}{b}$
\cite{MR4073883}.
It also defines naturally 
the $q$-analog of all real numbers as an infinite series over the variable $q$
\cite{MorierGenoud2019}.
The approach is based on the following $q$-deformation of the generators
$R=\left(\begin{smallmatrix}
    1&1\\
    0&1
\end{smallmatrix}\right)$
and
$S=\left(\begin{smallmatrix}
    0&-1\\
    1&0
\end{smallmatrix}\right)$
of $\mathrm{PSL}_2(\Z)=\SL_2(\Z)/\pm\mathrm{Id}$:
\[
    R_q=\left(\begin{array}{cc}
        q & 1 \\
        0 & 1
    \end{array}\right)
    \qquad
    \text{ and }
    \qquad
    S_q=\left(\begin{array}{cc}
        0 & -q^{-1} \\
        1 & 0
    \end{array}\right).
\]
Since $\mu(\a)=R^2SR$
and $\mu(\b)=R^3SR^2SR$, 
the $q$-analog of
$\mu(\a)$ and $\mu(\b)$ are 
\cite{leclere_q-deformations_2021}
\begin{align*}
    \mu_q(\a)&=R_q^2S_qR_q=
\left(\begin{array}{cc}
    q + q^{2} & 1 \\
    q & 1
\end{array}\right),\\
    \mu_q(\b)&=R_q^3S_qR_q^2S_qR_q=
\left(\begin{array}{cc}
    q + 2q^2+q^3+q^4 & 1 + q \\
    q + q^{2} & 1
\end{array}\right).
\end{align*}
Therefore, this defines a morphism of monoids $\mu_q:\{\a,\b\}^*\to\GL_2(\Z[q^{\pm 1}])$.

The $q$-analog of a nonnegative integer, a rational or a real number $\alpha$ 
has the property of being equal to $\alpha$ when evaluated at $q=1$ or more generally when $q\to 1$.
Likewise, we may observe that $\mu_{q}(\a)|_{q=1}=\mu(\a)$ and $\mu_{q}(\b)|_{q=1}=\mu(\b)$. Therefore, we have
$\mu_{q}(w)_{12}|_{q=1}=\mu(w)_{12}$ for all $w\in\{\a,\b\}^*$.
Thus, if $w\in\{\a,\b\}^*$ is a Christoffel word, the polynomial $\mu_q(w)_{12}$ over the variable $q$
is a \emph{$q$-analog of a Markoff number} satisfying that $\mu_{q}(w)_{12}$ evaluated at $q = 1 $ is a Markoff number, see Figure~\ref{fig:tree-of-q-Markoff-numbers}.
For example,
\[
    \mu_q(\a\a\b\a\b)_{12}=
        1 + 4 q + 10 q^{2} + 18 q^{3} + 27 q^{4} + 33 q^{5} + 33 q^{6} + 29 q^{7} + 21 q^{8} + 12 q^{9} + 5 q^{10} + q^{11}.
\]
A natural question is to understand the structure of the coefficients of
$\mu_{q}(w)_{12}$ whose sum is a Markoff number when $w$ is a Christoffel word.

\begin{figure}[h]
\begin{center}
    \begin{tikzpicture}[yscale=.8]
    \node[draw,right,align=center] (A0) at (0,0) {$\a.\b$\\   \scriptsize $1+q+2q^2+q^3$};
    \node[draw,right,align=center] (B0) at (2,-3.5) {$\a.\a\b$\\ \scriptsize$1 + 2 q + 3 q^{2} + 3 q^{3} + 3 q^{4} + q^{5}$};
    \node[draw,right,align=center] (B1) at (2,3.5) {$\a\b.\b$\\  \scriptsize
         $1 + 2 q + 5 q^{2} + 6 q^{3} + 6 q^{4} + 5 q^{5} + 3 q^{6} + q^{7}$};
    %\tikzstyle{every node}=[font=\footnotesize]
    \node[draw,right,align=center] (C0) at (4,-5.5) {$\a.\a\a\b$\\  \scriptsize
        $1 + 3 q + 5 q^{2} + 7 q^{3} + 7 q^{4} + 6 q^{5} + 4 q^{6} + q^{7}$};
    \node[draw,right,align=center] (C1) at (4,-1.5) {$\a\a\b.\a\b$\\ \scriptsize
        $1 + 4 q + 10 q^{2} + 18 q^{3} + 27 q^{4} + 33 q^{5} + 33 q^{6} + 29 q^{7} + 21 q^{8} + 12 q^{9} + 5 q^{10} + q^{11}$};
    \node[draw,right,align=center] (C2) at (4,1.5) {$\a\b.\a\b\b$\\ \scriptsize
        $1 + 4 q + 12 q^{2} + 25 q^{3} + 42 q^{4} + 58 q^{5} + 68 q^{6} + 69 q^{7} + 61 q^{8} + 45 q^{9} + 28 q^{10} + 14 q^{11} + 5 q^{12} + q^{13}$};
    \node[draw,right,align=center] (C3) at (4,5.5) {$\a\b\b.\b$\\  \scriptsize
        $1 + 3 q + 9 q^{2} + 16 q^{3} + 24 q^{4} + 29 q^{5} + 29 q^{6} + 25 q^{7} + 18 q^{8} + 10 q^{9} + 4 q^{10} + q^{11}$};
    % edges
    \draw (A0) -- (B0);
    \draw (A0) -- (B1);
    \draw (B0) -- (C0);
    \draw (B0) -- (C1);
    \draw (B1) -- (C2);
    \draw (B1) -- (C3);
    \draw (C0) -- ++ (1,1);
    \draw (C0) -- ++ (1,-1);
    \draw (C1) -- ++ (1,1);
    \draw (C1) -- ++ (1,-1);
    \draw (C2) -- ++ (1,1);
    \draw (C2) -- ++ (1,-1);
    \draw (C3) -- ++ (1,1);
    \draw (C3) -- ++ (1,-1);
\end{tikzpicture}
\end{center}
    \caption{Binary tree of proper Christoffel words $w$ and proper $q$-Markoff
    numbers $\mu_q(w)_{12}$.
    The sequence of polynomials associated
    to words $w\in\{\a\a^*\b\}$ is a subsequence
    of a sequence indexed
    in the Online Encyclopedia of Integer Sequences
    at \url{http://oeis.org/A123245}.}
    \label{fig:tree-of-q-Markoff-numbers}
\end{figure}
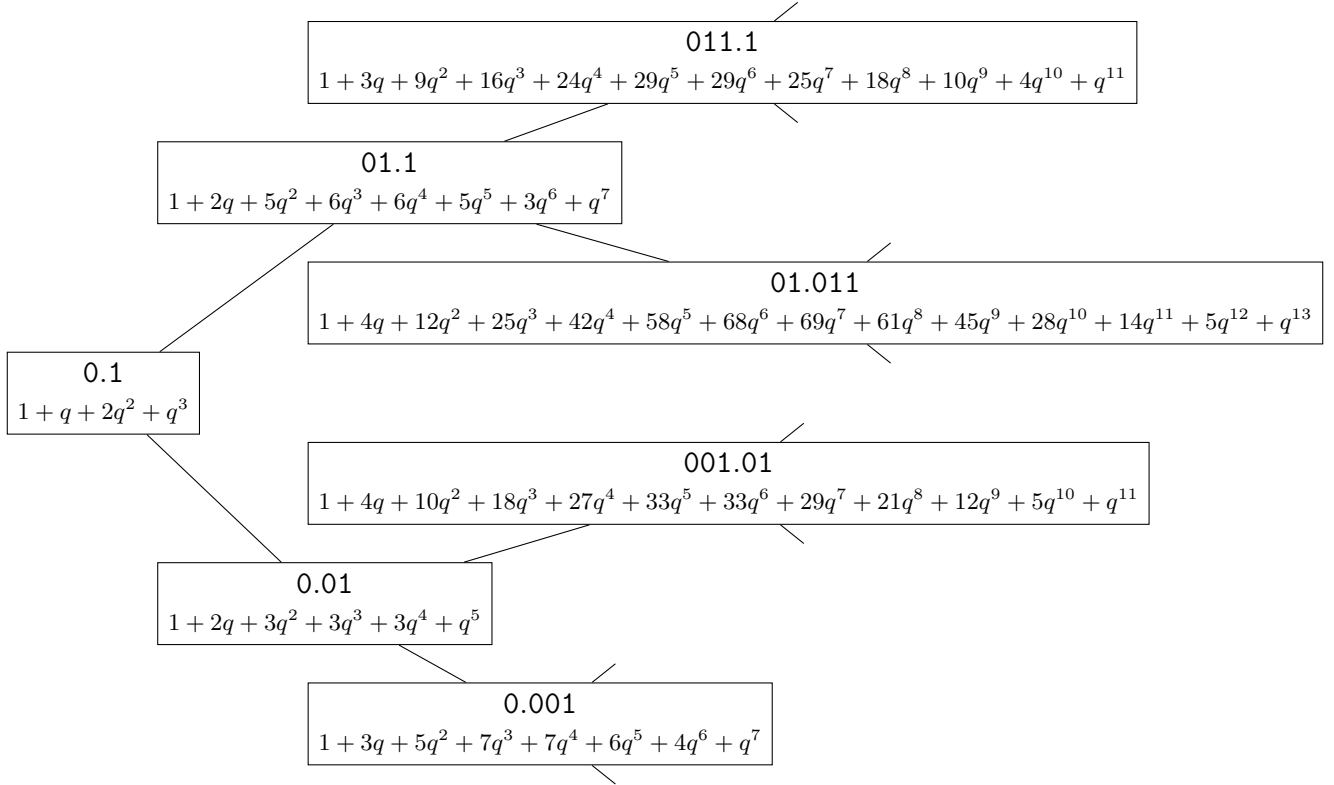

\begin{question}
    Is there a combinatorial interpretation for the degree and coefficients of
    the $q$-Markoff number $\mu_q(w)_{12}\in\Z[q]$ associated to a Christoffel word
    $w\in\{\a,\b\}^*$?
\end{question}

Also, a natural extension of the Markoff injectivity conjecture
to the $q$-analog of Markoff numbers is the following.

\begin{conjecture}[$q$-analog of the Markoff Injectivity Conjecture]
    \label{conj:q-analog-markoff}
    The map $\{\a,\b\}^*\to\Z[q]$ defined by
    $w\mapsto\mu_q(w)_{12}$ is injective over 
    the set of Christoffel words.
\end{conjecture}

% Elle est plus facile que la conjecture de Frobenius, car si celle-ci est vraie,
% la q-conjecture l'est aussi (puisque P_u(1) différent de P_v(1) implique P_u(q)
% différent de P_v(q) ).

Markoff Injectivity Conjecture implies \Cref{conj:q-analog-markoff},
since if 
$\mu_q(u)_{12}|_{q=1}\neq\mu_q(v)_{12}|_{q=1}$,
then $\mu_q(u)_{12}\neq\mu_q(v)_{12}$.
However, \Cref{conj:q-analog-markoff} is weaker than the classical Markoff
Injectivity Conjecture since two different polynomials may take the same value
at $q = 1$. So a priori \Cref{conj:q-analog-markoff} may be ``easier'' than
the classical conjecture. 

As mentioned above, the Markoff Injectivity Conjecture was extended to the
language of factors of all Christoffel words \cite[Conjecture 2]{LR2021}.
This language is equal to the language of all balanced sequences over a binary alphabet.
Balanced sequences include biinfinite periodic $\leftidx{^\infty}w^\infty$
repetitions of a Christoffel word, 
Sturmian sequences which are aperiodic
and more (ultimately periodic biinfinite words which are not purely periodic, called skew by Morse and Hedlund \cite{MR0000745}). 
See its definition in Section~\ref{sec:balanced}.
We extend the $q$-analog of the Markoff Injectivity Conjecture
to the language of all balanced sequences
$\Bcal=\{s\in\{\a,\b\}^\Z\colon s \text{ is balanced}\}$.

\begin{conjecture}
    \label{conj:q-analog-markoff-over-balanced-language}
    The map $\{\a,\b\}^*\to\Z[q]$ defined by
    $w\mapsto\mu_q(w)_{12}$ is injective over 
    the language 
    $\Lcal(\Bcal)=\bigcup_{s\in\Bcal}\Lcal(s)$
    of all balanced sequences.
\end{conjecture}

\Cref{conj:q-analog-markoff-over-balanced-language}
implies
\Cref{conj:q-analog-markoff} since the set of Christoffel words is a subset of $\Lcal(\Bcal)$.

\subsection{Main results}

In this article, 
we propose a result which progresses in the direction of
\Cref{conj:q-analog-markoff-over-balanced-language}.
More precisely, we prove that the map $w\mapsto\mu_q(w)_{12}$ is strictly increasing
with respect to the radix order on the language of a fixed balanced sequence.

Recall that the radix order is defined
for every $u,v\in\{\a,\b\}^*$ as
\[
    u<_{radix}v 
    \qquad
    \text{ if } 
    \qquad
\begin{cases}
    |u|<|v| \quad\text{ or }\\
    |u|=|v| \quad\text{ and } \quad u<_{lex} v.
\end{cases}
\]
Also we define a partial order $\prec$ on $\Z[q]$ as
\[
    f \prec g
    \quad
    \text{ if and only if }
    \quad
    f \neq g
    \text{ and }
    g - f \in \Z_{\geq0}[q].
\]

\begin{maintheorem}\label{thm:main-result-intro}
    Let $s\in\{\a,\b\}^\Z$ be a balanced sequence
    and $u,v\in\Lcal(s)$ be two factors in the language of $s$.
    If $u<_{radix}v$, then 
    $\mu_q(u)_{12} \prec \mu_q(v)_{12}$,
    i.e., $\mu_q(v)_{12} - \mu_q(u)_{12}$
    is a nonzero polynomial of indeterminate $q$ with nonnegative integer
    coefficients.
\end{maintheorem}

As a consequence, we prove 
\Cref{conj:q-analog-markoff-over-balanced-language}
when restricted to the language of a given balanced sequence.

\begin{maincorollary}\label{cor:main-corollary-intro-injective}
    Let $s\in\{\a,\b\}^\Z$ be a balanced sequence.
    The map $u\mapsto\mu_q(u)_{12}$ is injective over the language $\Lcal(s)$.
\end{maincorollary}

Remark that \Cref{cor:main-corollary-intro-injective} can also be deduced from
Corollary 1 of \cite{LR2021} since two polynomials evaluated at $q=1$ are
distinct implies that the polynomials are distinct.
We also state a corollary of \Cref{thm:main-result-intro} when evaluating
polynomials at $q=\gamma$ for all positive real numbers $\gamma>0$
improving Corollary 1 from \cite{LR2021}.

\begin{maincorollary}\label{cor:main-corollary-intro-q>0}
    Let $s\in\{\a,\b\}^\Z$ be a balanced sequence.
    For every $\gamma>0$,
    the map 
    $\{\a,\b\}^*\to\R$ defined by
    $w\mapsto\mu_q(w)_{12}|_{q=\gamma}$ is strictly increasing and injective over the
    language $\Lcal(s)$ with respect to the radix order $<_{radix}$.
\end{maincorollary}

The proof of Theorem~\ref{thm:main-result-intro} follows the same idea  
as the proof that the map $\{\a,\b\}^*\to\Z_{\geq0}:w\mapsto\mu(w)_{12}$ 
is strictly increasing (for the radix order) over the factors of a Christoffel word
\cite{lapointe_these_2020,LR2021}.
When listing the conjugates of a Christoffel word in lexicographic order,
only a flip of two letters happens between consecutive conjugates
\cite[Corollary 5.1]{MR2197281}.
This observation was done in the context of the Burrows-Wheeler transform
\cite{MR1976388}.
Recall that Burrows-Wheeler transform of a finite word $w$ is obtained from
$w$ by first listing the conjugates of $w$ in lexicographic order and
then concatenating the final letters of the conjugates in this
order, see \cite{MR2364562}.
When listing lexicographically the $n+1$ factors of a balanced language, 
at most two letters are changed from one word to the next (see
Lemma~\ref{lem:cyclic-ordering-sturmian-factors}). This allows to prove
Theorem~\ref{thm:main-result-intro} for the language of a balanced biinfinite sequence.

Table~\ref{T:fibo_factor} in the appendix shows the
values of $\mu(w)_{12}$ and $\mu_q(w)_{12}$ for the small factors in the
Fibonacci word. We observe that the coefficients of the polynomials over $q$
are increasing from one row to the next.
To each factor $w$ corresponds a function $\gamma\mapsto\mu_q(w)_{12}|_{q=\gamma}$.
The graph of these functions is shown on the interval $0<\gamma<100$ 
in \Cref{fig:log_polynomials}
for every short factors $w$ in Fibonacci word.

% Image log_polynomials.pdf is created by:
% sage log_polynomials.sage

\begin{figure}[h]
\begin{center}
    \includegraphics{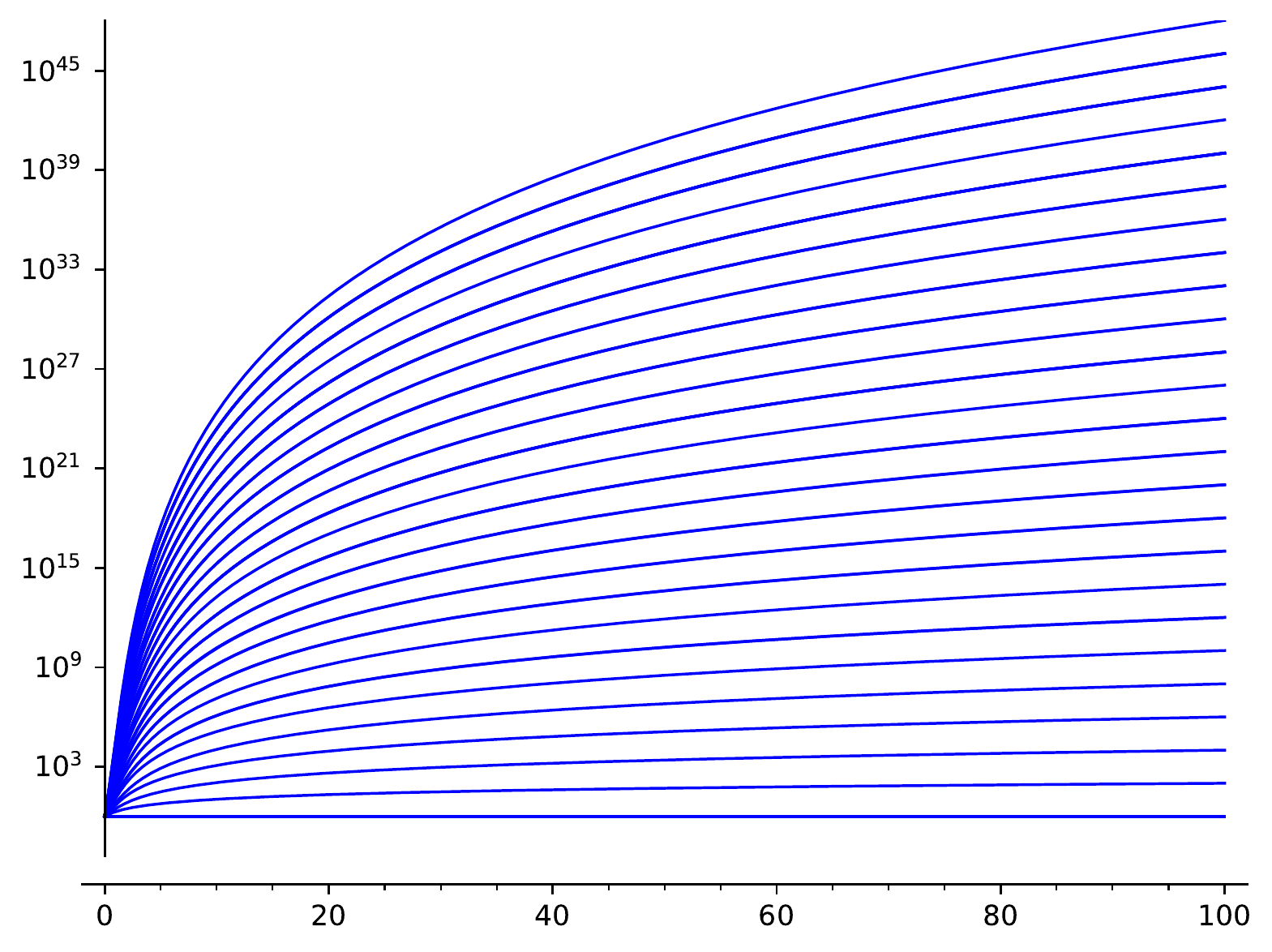}
\end{center}
    \caption{The graph of the curves 
    $(\gamma,\mu_q(w)_{12}|_{q=\gamma})$ for $0\leq\gamma\leq 100$ for all
    55 factors $w$ of length $|w|<10$ in the Fibonacci word.
    %Sturmian subshift of slope
    %$1/\varphi^2$ where $\varphi=\frac{1+\sqrt{5}}{2}$. 
    The 55 polynomials $\mu_q(w)_{12}$ shown in the appendix
    are ranging from degree 0 to 24 explaining why we see 
    25 curves in the figure instead of 55.
    A consequence of \Cref{thm:main-result-intro}
    is that the 55 curves do not intersect when $\gamma>0$.
    The difference between two polynomials
    of the same degree is very small
    and can not be distinguished (the $y$ axis uses a logarithmic scale).
    For example, there are four factors of length 9
    whose images under $w\mapsto\mu_q(w)_{12}$ are
    polynomials of degree 23 whose pairwise
    difference is a polynomial of degree 20, 18 or 14, see
    \Cref{T:fibo_factor}.}
    \label{fig:log_polynomials}
\end{figure}

This article is structured as follows. 
\Cref{sec:balanced} gathers many equivalent characterizations of balanced
sequences including properties introduced by Markoff 
\cite{M1879,M1882}.
It is then related to indistinguishable asymptotic pairs \cite{MR4227737}
which we state as \Cref{thm:equivalence-indistinguishable-pair}
since it provides a link between an old notion of Markoff with a recent one.
Indistinguishable asymptotic pairs naturally
provides two compact representations of the language of
length $n$ of a balanced sequence as the factors appearing in two words of
length $2n$, see \Cref{cor:hamiltonian-sturmian}.
This is used to show that only small local changes appear between a factor and
the next factor according to the radix order over the language of a balanced
sequence.
In \Cref{sec:increasing-small-changes}, we show that the map
$w\mapsto\mu_q(w)_{12}$ is increasing over the listed small local changes.
The proof of 
\Cref{thm:main-result-intro}
is done in
\Cref{sec:proof-main-results}.
In \Cref{sec:conclusion}, we conclude with few examples illustrating
that
\Cref{thm:main-result-intro}
can unfortunately not be extended to the language of all balanced sequences
using the radix order.

\subsection*{Acknowledgements}
This work was initiated during the 
\emph{Journées de combinatoire de Bordeaux}
(\href{https://jcb2021.labri.fr/}{JCB 2021})
held online February $1^{st}$--$4^{th}$ 2021, after the talks
made by Sophie Morier-Genoud and Mélodie Lapointe.
We are thankful to the anonymous referees for their insightful comments leading
to considerable improvements in the presentation of this article.

\section{Balanced sequences}
\label{sec:balanced}

\subsection{Definition and example}

% sage: p = words.FibonacciWord('ab')[:20]
% sage: p
% word: abaababaabaababaabab
% sage: p.reversal()
% word: babaababaabaababaaba

% We say that a language $\Lcal\subset\Sigma^*$ 
% over a finite set $\Sigma$
% is \emph{factorial} whenever
% $pus\in\Lcal$ implies that $u\in\Lcal$ for all $p,u,s\in\Sigma^*$.
Let $s\in\Sigma^\Z$ be a sequence over a finite set $\Sigma$.
The language of $s$ is $\Lcal(s)=\{s_ks_{k+1}\cdots s_{k+n-1}\mid k\in\Z, n\geq 0\}\subset\Sigma^*$
is the set of subwords (or factors) occurring in $s$.
The \emph{reversal} of a finite word $w = (w_i)_{1 \leq i \leq n}$ is $\widetilde{w} = (w_{n+1-i})_{1 \leq i \leq n}$. Similarly, the \emph{reversal} of a right infinite sequence $s = (s_i)_{i \in \Z_{\geq 0}}$ is the left infinite sequence $\widetilde{s} = (s_{-i})_{i \in \Z \leq 0}$.
\begin{definition}
A sequence $s\in\Sigma^\Z$ is \emph{balanced} if for every positive integer $n$,
for every $u,v\in\Lcal(s)\cap\Sigma^n$
and every letter $a\in\Sigma$,
the number of $a$'s occurring in $u$ and $v$ differ by at most 1.
\end{definition}
For example, the right-infinite Fibonacci word 
\[
    F =\a\b\a\a\b\a\b\a\a\b\a\a\b\a\b\a\a\b\a\b \ldots \in\Sigma^{\Z_{\geq0}}
\]
over the alphabet $\Sigma=\{\a,\b\}$ is 
such that both
\[
\widetilde{F}\cdot \a\b\cdot F = 
\ldots
\b\a\b\a\a\b\a\b\a\a\b\a\a\b\a\b\a\a\b\a\cdot
\a\b\cdot
\a\b\a\a\b\a\b\a\a\b\a\a\b\a\b\a\a\b\a\b
\ldots
\]
and
\[
\widetilde{F}\cdot \b\a\cdot F = 
\ldots
\b\a\b\a\a\b\a\b\a\a\b\a\a\b\a\b\a\a\b\a\cdot
\b\a\cdot
\a\b\a\a\b\a\b\a\a\b\a\a\b\a\b\a\a\b\a\b
\ldots
\]
are balanced sequences.
This observation is illustrated for factors of length up to six in the
following table.
\[
    \begin{array}{llp{13mm}p{13mm}}
        n & \Lcal(\widetilde{F}\a\b F)\cap\Sigma^n  & number of $\a$'s& number of $\b$'s\\
        \hline
        0 & \{\varepsilon\}                              & 0            & 0            \\
        1 & \{\a, \b\}                                     & 0\text{ or }1& 0\text{ or }1\\
        2 & \{\a\a, \a\b, \b\a\}                               & 1\text{ or }2& 0\text{ or }1\\
        3 & \{\a\a\b, \a\b\a, \b\a\a, \b\a\b\}                       & 1\text{ or }2& 1\text{ or }2\\
        4 & \{\a\a\b\a, \a\b\a\a, \a\b\a\b, \b\a\a\b, \b\a\b\a\}             & 2\text{ or }3& 1\text{ or }2\\
        5 & \{\a\a\b\a\a, \a\a\b\a\b, \a\b\a\a\b, \a\b\a\b\a, \b\a\a\b\a, \b\a\b\a\a\} & 3\text{ or }4& 1\text{ or }2\\
        6 & \{\a\a\b\a\a\b, \a\a\b\a\b\a, \a\b\a\a\b\a, \a\b\a\b\a\a, \b\a\a\b\a\a, \b\a\a\b\a\b, \b\a\b\a\a\b\}
                                                         & 3\text{ or }4& 2\text{ or }3\\
    \end{array}
\]
From the table, we confirm that the number of $\a$'s and the number of $\b$'s
occurring in two factors of the same length differ by at most 1.

\subsection{The Markoff property} 

It is worth recalling that balanced sequences appeared in the work of Markoff
himself \cite{M1879,M1882} under an equivalent condition, called Markoff property (M) in
\cite{MR2242616}.

\begin{definition}
    {\rm\cite{MR2242616}}
We say that 
a biinfinite word
$s\in\{\a,\b\}^\Z$
satisfies the \emph{Markoff property} if for any factorization
$s=uxyv$, where $\{x,y\}=\{\a,\b\}$, one has
\begin{itemize}
    \item either $\widetilde{u} = v$,
    \item or there is a factorization $u=u'ym$, $v=\widetilde{m}xv'$.
\end{itemize}
\end{definition}
The Markoff property is related to the Markoff spectrum.
Let $U=(a_i)_{i\in\Z}$ be a biinfinite sequence such that $a_i$ are positive integers. 
For $i\in\Z$, let
\[
    \lambda_i(U) = a_i + [0;a_{i+1},a_{i+2},\dots]
                       + [0;a_{i-1},a_{i-2},\dots].
\]
The \emph{Markoff supremum of $U$} is
\[
    M(U) = \sup_{i\in\Z}\lambda_i(U).
\]
Two results of Markoff can be stated in terms of Christoffel
words and balanced sequences as follows where
$\sigma$ is the substitution from $\{\a,\b\}^*$ to $\{1,2\}^*$ defined by $\a\mapsto 11$ and $\b\mapsto 22$.
It provides an equivalence between sequences satisfying the
Markoff property and sequences of positive integers such that the Markoff
supremum is at most~3.
The equivalence between sequences satisfying the Markoff property and
balanced sequences was not proved by Markoff himself:
it was stated without proof in \cite{MR1010419} and a proof was provided in
\cite{MR2242616}.

% Pour le th. 2.3, il faut que je clarifie les choses: l'équivalence de 1
% (condition appelée "Markoff balanced" par Cusick et Flahive) et 3 est due à
% Markoff, et la condition équivalente 2 (balanced) est indiquée sans preuve dans
% le livre de Cusick et Flahive; ma contributionn était d'en donner une preuve.

\begin{theorem}[Markoff]
    {\rm\cite[Theorem 3.1 and 7.1]{MR2242616}}
Let $s\in\{\a,\b\}^\Z$ be a biinfinite word. The following conditions are equivalent:
    \begin{itemize}
        \item $s$ satisfies the Markoff property,
        \item $s$ is balanced,
        \item $M(\sigma(s))\leq 3$.
    \end{itemize}
\end{theorem}

The Markoff supremum of a purely periodic balanced sequence can be computed
from the Markoff number associated to the Christoffel word which is a period of
the sequence.

\begin{theorem}[Markoff]
    {\rm\cite[Theorem 6.2.1]{MR3887697}}
    Let $w$ be some lower Christoffel word associated with Markoff number
    $m=\mu(w)_{12}$.
    Let $s$ be the biinfinite sequence $\leftidx{^\infty}\sigma(w)^\infty$.
    Then $M(s)=\sqrt{9-\frac{4}{m^2}}$.
\end{theorem}

\subsection{Mechanical sequences}

It is known that right-infinite aperiodic balanced sequences correspond to
mechanical sequences \cite{MR0000745} which are binary encodings of irrational
rotations, see the chapters \cite[Chapter 6]{MR1970385}, 
\cite[Chapter 2]{MR1905123} and \cite[Chapter 9]{MR1997038}.
A biinfinite balanced sequence can also be periodic and in this case expressed in terms of Christoffel words.
To be more precise, let $\alpha\in [0,1]$ and $\rho\in\R$ and consider
the \emph{lower} and \emph{upper mechanical sequences} $s_{\alpha,\rho}$ and 
$s'_{\alpha,\rho}$ 
with \emph{slope} $\alpha$ and \emph{intercept} $\rho$
given respectively by
\[
\begin{array}{rccl}
s_{\alpha,\rho}:&\Z & \to & \{\symb{0},\symb{1}\}\\
&n & \mapsto &\lfloor\alpha(n+1)+\rho\rfloor-\lfloor\alpha n+\rho\rfloor
\end{array}
\quad
\text{ and }
\quad
\begin{array}{rccl}
s'_{\alpha,\rho}:&\Z & \to & \{\symb{0},\symb{1}\}\\
&n & \mapsto &\lceil\alpha(n+1)+\rho\rceil-\lceil\alpha n+\rho\rceil.
\end{array}
\]
When $\alpha$ is rational, the sequences
$s_{\alpha,\rho}$ and $s'_{\alpha,\rho}$ are periodic 
and their period corresponds to a Christoffel word \cite{MR2464862}.
When $\alpha$ is irrational, then $s_{\alpha,\rho}$ and $s'_{\alpha,\rho}$ 
are not periodic.
    It is clear that if $\rho-\rho'$ is an integer, then
    $s_{\alpha,\rho}=s_{\alpha,\rho'}$ and
    $s'_{\alpha,\rho}=s'_{\alpha,\rho'}$. Thus we may always assume
    $0\leq\rho<1$.
    Moreover,
    if $\Z\cap\alpha\Z+\rho=\varnothing$ then
    $s_{\alpha,\rho}=s'_{\alpha,\rho}$.

\subsection{Four classes of balanced sequences}

    Biinfinite balanced sequences can be split into four different types of
    sequences.
    Reutenauer proposed the following refinement of the Markoff property
    \cite{MR2242616} which was restated in \cite{MR2387854} as follows.
    If a biinfinite sequence $u\in\{\a,\b\}^\Z$ satisfies the Markoff property, then it falls into exactly one of the following classes:
\begin{itemize}
    \item [$(M_1)$] 
        $u$ cannot be written as $u=\widetilde{p}xyp$
        where $\{x,y\}=\{\a,\b\}$
        and
        the lengths of the Christoffel words occurring
        in $u$ are bounded;
    \item [$(M_2)$] 
        $u$ cannot be written as $u=\widetilde{p}xyp$ 
        where $\{x,y\}=\{\a,\b\}$
        and
        the lengths of the Christoffel words occurring
        in $u$ are unbounded;
    \item [$(M_3)$] $u$ has a unique factorization $u=\widetilde{p}xyp$
        where $\{x,y\}=\{\a,\b\}$;
    \item [$(M_4)$] $u$ has at least two factorizations $u=\widetilde{p}xyp$
        where $\{x,y\}=\{\a,\b\}$.
\end{itemize}

Morse and Hedlund gave a classification of balanced biinfinite sequences
into three classes (periodic, Sturmian, skew) \cite{MR0000745}.
Since the Sturmian case naturally splits into two, Reutenauer proposed the
following four classes $(MH_i)_{i\in\{1,2,3,4\}}$ 
and proved their equivalence with the $(M_i)$.

\begin{theorem}\label{thm:balanced-4-cases}
    {\rm\cite[Theorem 6.1]{MR2242616}}
    Let $u\in\{\a,\b\}^\Z$ be a balanced sequence. 
    For every $i\in\{1,2,3,4\}$,
    $u$ satisfies $(M_i)$ if and only if
    $u$ satisfies $(MH_i)$
    where
\begin{itemize}
    \item [$(MH_1)$] $u$ is a purely periodic word
        $\leftidx{^\infty}w^\infty$ for some Christoffel word $w$,
    \item [$(MH_2)$] $u$ is a generic aperiodic Sturmian word, 
        i.e., $u=s_{\alpha,\rho}=s'_{\alpha,\rho}$ 
        for some $\alpha\in[0,1]\setminus\Q$ and
        $\rho\in\R$ such that $\Z\cap\alpha\Z+\rho=\varnothing$.
    \item [$(MH_3)$] $u$ is a characteristic aperiodic Sturmian word,
        i.e., $u=s_{\alpha,\rho}$ or $u=s'_{\alpha,\rho}$ 
        for some $\alpha\in[0,1]\setminus\Q$ and
        $\rho\in\R$ such that $\Z\cap\alpha\Z+\rho\neq\varnothing$.
    \item [$(MH_4)$] $u$ is an ultimately periodic word but not
        purely periodic, i.e.,
        $u=\cdots xxyxx \cdots$ or $u=\cdots (ymx)(ymx)(ymy)(xmy)(xmy)\cdots$
        where $\{x,y\}=\{\a,\b\}$ and $\a m\b$ is a Christoffel word.
\end{itemize}
\end{theorem}

\subsection{Indistinguishable asymptotic pairs}

In this section, we give equivalent conditions for
balanced sequences satisfying cases $(M_3)$ or $(M_4)$.
Cases $(M_3)$ and $(M_4)$ can be expressed in terms of
limits of mechanical words toward an irrational or rational slope from above or
from below which were shown to be equivalent to sequences that belong to an
indistinguishable asymptotic pair \cite{MR4227737}.

    Concretely, given a finite set $\Sigma$, we consider the space of
    sequences $\Sigma^{\Z} = \{ s \colon \Z \to \Sigma\}$ endowed with
    the prodiscrete topology and 
    the \emph{shift} $\sigma \colon \Sigma^{\Z}\to
    \Sigma^{\Z}$ where 
    \[ 
    \left(   \sigma(s) \right)_m  =  s_{m+1} \quad \mbox{ for every } 
             m \in \Z \mbox{ and } s \in \Sigma^{\Z}.  
    \]
    The shift on $\Sigma^{\Z}$ is invertible and extends to
    a shift action $\Z
    \overset{\sigma}{\curvearrowright} \Sigma^{\Z}$. 
    In this setting, two
    sequences $s,t \in \Sigma^{\Z}$ are \emph{asymptotic} if $s$ and
    $t$ differ in finitely many positions of $\Z$. The finite set $F = \{ n \in
    \Z : s_n \neq t_n\}$ is called the \emph{difference set} of
    $(s,t)$.

    Given two asymptotic sequences $s,t \in \Sigma^{\Z}$, we may compare the
    number of occurrences of a fixed pattern.  A \emph{pattern} is a function
    $p \colon S \to \Sigma$ where $S$, called
    \emph{support}, is a finite subset of $\Z$.
    The \emph{occurrences} of a pattern $p\in\Sigma^S$ in a sequence
    $s\in\Sigma^\Z$ is the set $\occ_p(s):=\{n\in\Z\colon \sigma^n(s)|_S = p\}$.
    Observe that when $s,t \in \Sigma^{\Z}$ are asymptotic sequences,
    the difference $\occ_p(s)\setminus \occ_p(t)$ is finite
    because the occurrences of $p$ are the same outside the difference set.
    We say that $(s,t)$ is
    an \emph{indistinguishable asymptotic pair} 
    if $(s,t)$ is asymptotic and
    \[
        \#\left(\occ_p(s)\setminus \occ_p(t)\right) = 
        \#\left(\occ_p(t)\setminus \occ_p(s)\right)
    \] 
    for every finite support $S\subset\Z$ and
    every pattern $p\in\Sigma^S$. 
    Extending the results proved in \cite{MR4227737} about indistinguishable
    asymptotic pairs, we may prove equivalent conditions for balanced sequences
    satisfying Markoff property $(M_3)$ or $(M_4)$.
    In the statement, we denote the position of the origin of a biinfinite
    sequence $s=\cdots s_{-2}s_{-1}.s_0s_1s_2\cdots\in\Sigma^\Z$ with a dot
    ($.$) between positions $-1$ and $0$.

\begin{maintheorem}\label{thm:equivalence-indistinguishable-pair}
    Let $s\in\{\a,\b\}^{\Z}$ and $n_0\in\Z$.
    The following are equivalent conditions describing balanced sequences
    satisfying Markoff property $(M_3)$ or $(M_4)$:
    \begin{enumerate}
        \item the sequence $s$ has a factorization $\sigma^{n_0}s=\widetilde{p}x.yp$ where
            $\{x,y\}=\{\a,\b\}$;
        \item there exists a sequence $(\alpha_k)_{k\in\Z_{\geq0}}$ with $\alpha_k\in[0,1]\setminus\Q$
            such that $\sigma^{n_0}s=\lim_{k\to\infty} s_{\alpha_k,0}$
            or        $\sigma^{n_0}s=\lim_{k\to\infty} s'_{\alpha_k,0}$;
        \item there exists a sequence $t\in\{\a,\b\}^{\Z}$
            such that $(s,t)$ is 
            an indistinguishable asymptotic pair
            with difference set $\{n_0-1,n_0\}$.
    \end{enumerate}
\end{maintheorem}

\begin{proof}
    (1) $\implies$ (2).
    It is sufficient to prove it for $n_0=0$.
    We suppose that $s$ has a factorization
    $s=\widetilde{p}x.yp$
    where $\{x,y\}=\{\a,\b\}$.
    The symmetry $n\mapsto -n-1$ keeps the sequence $s$ invariant
    except at $\{-1,0\}$. In other words,
    $s(n)=s(-n-1)$ for every $n\in\Z\setminus\{-1,0\}$
    and $\{s(-1),s(0)\}=\{\a,\b\}$.

    Suppose that $s$ satisfies case $(M_3)$.
    From \Cref{thm:balanced-4-cases}, $s$ also satisfies case $(MH_3)$,
    that is, there exists an irrational number
    for some $\alpha\in[0,1]\setminus\Q$ and
    $\rho\in\R$ such that
        $s=s_{\alpha,\rho}$ or $s=s'_{\alpha,\rho}$ 
        with $\Z\cap\alpha\Z+\rho\neq\varnothing$.
        Since $s(n)=s(-n-1)$ for every $n\in\Z\setminus\{-1,0\}$,
        we must have $\rho=0$.
        Because $\alpha$ is irrational, 
        for every $n\in\Z\setminus\{-1,0\}$ and
        every sequence $(\alpha_k)_{k\in\Z_{\geq0}}$ 
        of irrational numbers $\alpha_k$
        such that $\lim_{k\to\infty}\alpha_k=\alpha$,
        we have
        \[
        s_{\alpha,0}(n)
        =\lfloor\alpha(n+1)\rfloor-\lfloor\alpha n\rfloor
        =\lim_{k\to\infty}\lfloor\alpha_k(n+1)\rfloor-\lfloor\alpha_k n\rfloor
        =\lim_{k\to\infty} s_{\alpha_k,0}(n),
        \]
        \[
        s'_{\alpha,0}(n)
        =\lceil\alpha(n+1)\rceil-\lceil\alpha n\rceil
        =\lim_{k\to\infty}\lceil\alpha_k(n+1)\rceil-\lceil\alpha_k n\rceil
        =\lim_{k\to\infty} s'_{\alpha_k,0}(n).
        \]
        Also, since each $\alpha_k$ is irrational, we have
        $s_{\alpha,0}(0)s_{\alpha,0}(1)
         =\b\a=\lim_{k\to\infty}s_{\alpha_k,0}(0)s_{\alpha_k,0}(1)$
         and
        $s'_{\alpha,0}(0)s'_{\alpha,0}(1)
         =\a\b=\lim_{k\to\infty}s'_{\alpha_k,0}(0)s'_{\alpha_k,0}(1)$.
         We conclude that
            $s=\lim_{k\to\infty} s_{\alpha_k,0}$
            or $s=\lim_{k\to\infty} s'_{\alpha_k,0}$.

    Suppose that $s$ satisfies case $(M_4)$. 
    From \Cref{thm:balanced-4-cases}, $s$ also satisfies case $(MH_4)$,
    that is,
     $s$ is an ultimately periodic word but not
        purely periodic, i.e.,
        $s=\cdots xxyxx \cdots$ or $s=\cdots (ymx)(ymx)(ymy)(xmy)(xmy)\cdots$
        where $\{x,y\}=\{\a,\b\}$ and $\a m\b$ is a Christoffel word.
        From 
        \cite[Lemma 4.2]{MR4227737}, there exists $a,b\in\Z_{\geq0}$ coprime
        integers such that 
        \[
            \widetilde{p}\b.\a p =\lim_{\alpha\to\frac{a}{a+b}^+}s_{\alpha,0}
        \quad\text{ and }
        \quad
        \widetilde{p}\a.\b p =\lim_{\alpha\to\frac{a}{a+b}^+}s'_{\alpha,0}
        \quad
        \text{(limit from above)}
        \]
        or
        \[
            \widetilde{p}\b.\a p =\lim_{\alpha\to\frac{a}{a+b}^-}s_{\alpha,0}
        \quad\text{ and }
        \quad
        \widetilde{p}\a.\b p =\lim_{\alpha\to\frac{a}{a+b}^-}s'_{\alpha,0}
        \quad
        \text{(limit from below).}
        \]

    (2) $\implies$ (1). 
    If $\lim_{k\to\infty}\alpha_k\in\Q$, then
    from \cite[Lemma 4.2]{MR4227737}, we directly have that
        $\sigma^{n_0}s$ has a factorization $\sigma^{n_0}s=\widetilde{p}x.yp$ where
            $\{x,y\}=\{\a,\b\}$.
    If $\lim_{k\to\infty}\alpha_k=\alpha\in[0,1]\setminus\Q$, then
            $\lim_{k\to\infty} s_{\alpha_k,0}=s_{\alpha,0}$
            and
            $\lim_{k\to\infty} s'_{\alpha_k,0}=s'_{\alpha,0}$.
            Both are symmetric satisfying
            $s_{\alpha,0}(n)=s_{\alpha,0}(-n-1)$ 
            and
            $s'_{\alpha,0}(n)=s'_{\alpha,0}(-n-1)$ 
            for every $n\in\Z\setminus\{-1,0\}$ and
            $\{s_{\alpha,0}(-1),s_{\alpha,0}(0)\}
            =\{s'_{\alpha,0}(-1),s'_{\alpha,0}(0)\}
            =\{\a,\b\}$.

    (2) $\iff$ (3).
    It was proved in \cite[Theorem B]{MR4227737} that (2) holds
    if and only if there exists a sequence $t'\in\{\a,\b\}^{\Z}$ such that
    $(\sigma^{n_0}s,t')$ is an indistinguishable asymptotic pair with
    difference set $\{-1,0\}$. 
    This holds
    if and only if 
    $(s,\sigma^{-n_0}t')$ is an indistinguishable asymptotic pair with
    difference set $\{n_0-1,n_0\}$ since
    the shift preserves indistinguishable asymptotic pairs
    \cite[Proposition 2.5]{MR4227737}. 
    It concludes the proof if we let $t=\sigma^{-n_0}t'$.
\end{proof}

\subsection{Language of a balanced sequence}

% Let $\alpha\in [0,1]$ and consider
% the lower and upper sequences $c_{\alpha}$ and $c'_{\alpha}$ given
% respectively by
% \[
% \begin{array}{rccl}
% c_{\alpha}:&\Z & \to & \{\symb{0},\symb{1}\}\\
% &n & \mapsto &\lfloor\alpha(n+1)\rfloor-\lfloor\alpha n\rfloor
% \end{array}
% \quad
% \text{ and }
% \quad
% \begin{array}{rccl}
% c'_{\alpha}:&\Z & \to & \{\symb{0},\symb{1}\}\\
% &n & \mapsto &\lceil\alpha(n+1)\rceil-\lceil\alpha n\rceil.
% \end{array}
% \]
% We observe that $c_{\alpha}=s_{\alpha,0}$ and $c'_{\alpha}=s'_{\alpha,0}$.

Balanced sequences have other equivalent definitions, for example, in terms of
factor complexity \cite{MR0322838}. A balanced sequence satisfies Markoff properties
$(M_2)$, $(M_3)$ or $(M_4)$ if and only if it has complexity $n+1$, see
\cite[Theorem 2.1.13]{MR1905123} stated for right-infinite sequences.

The language of a balanced sequence of complexity $n+1$ can be compactly
represented in two ways as described in the following result.
It shows that there exist two words of length $2n$ that contains the language
of factors of length $n\geq1$ occurring in a balanced sequence.
The proof follows easily from the notion of indistinguishable asymptotic pairs
\cite{MR4227737}.

\begin{corollary}\label{cor:hamiltonian-sturmian}
Let $s\in\{\a,\b\}^\Z$ be a balanced sequence
having at least one factorization $s=\widetilde{p}xyp$
        where $\{x,y\}=\{\a,\b\}$.
        Let $n\geq1$ and $w$ be the prefix of $p$ of length $n-1$.
        The two words
$\widetilde{w}\a\b w$
and
$\widetilde{w}\b\a w$
of length $2n$
contain the $n+1$ factors of $s$. More precisely,
$
\Lcal_n(s) 
= \Lcal_n(\widetilde{w}\a\b w)
= \Lcal_n(\widetilde{w}\b\a w)$.
\end{corollary}

%Let $\alpha\in [0,1]$ and consider
%the biinfinite Sturmian sequence $c_{\alpha}$
%\[
%\begin{array}{rccl}
%c_{\alpha}:&\Z & \to & \{\symb{0},\symb{1}\}\\
%&n & \mapsto &\lfloor\alpha(n+1)\rfloor-\lfloor\alpha n\rfloor.
%\end{array}
%\]

\begin{proof}
    Let $n_0\in\Z$ be such that
    $\sigma^{n_0}s=\widetilde{p}x.yp$.
    We may assume that $n_0=0$ and $\sigma^{n_0}s=s$ since shifting the
    sequence $s$ preserves its language.  From
    \Cref{thm:equivalence-indistinguishable-pair},
    there exists a sequence $t\in\{\a,\b\}^{\Z}$
    such that $(s,t)$ is an indistinguishable asymptotic pair
    with difference set $\{n_0-1,n_0\}=\{-1,0\}$.
    In particular, $s$ and $t$ are equal outside of the difference set, i.e.,
    $s|_{\Z\setminus\{-1,0\}}=t|_{\Z\setminus\{-1,0\}}$,
    and different on the difference set.
    Since the alphabet is binary, we must have $t_{-1}t_0=yx$.
    Therefore the sequence $t$ satisfies $t=\widetilde{p}y.xp$.
    Replacing $\{x,y\}$ by $\{\a,\b\}$, we have that the indistinguishable asymptotic pair
    is of the form $\{s,t\}=\{\widetilde{p}\a.\b p,\widetilde{p}\b.\a p\}$.
    Thus, for every prefix $w$ of length $n-1$ of $p$, we have
        $\{s_{-n}s_{-n+1}\cdots s_{n-1},
           t_{-n}t_{-n+1}\cdots t_{n-1}\}
           =\{\widetilde{w}\a\b w, \widetilde{w}\b\a w\}$.
    From \cite[Corollary 3.6]{MR4227737}, we have 
    \begin{align*}
    \Lcal_n(s) 
        &= \Lcal_n(s_{-n}s_{-n+1}\cdots s_{n-1}) 
         = \Lcal_n(t_{-n}t_{-n+1}\cdots t_{n-1})\\
        &= \Lcal_n(\widetilde{w}\a\b w)
         = \Lcal_n(\widetilde{w}\b\a w).\qedhere
    \end{align*}
\end{proof}

For example, the following two words of length 16 contain one occurrence of each factor of length 8 occurring in the Fibonacci word:
\[
    \begin{tabular}{l}
\texttt{1010010.01.0100101}\\
\texttt{1010010.10.0100101}
    \end{tabular}
\]
We use \Cref{cor:hamiltonian-sturmian}
in \Cref{lem:cyclic-ordering-sturmian-factors}
to show the existence of a bijection
$f:\Lcal_n(s)\to\Lcal_n(s)$ 
which is a cyclic permutation and 
having the property, except for two words in $\Lcal_n(s)$, 
of flipping a $\a\b$ into a $\b\a$ from $u$ to $f(u)$.
This property is well-known in the context of the Burrows-Wheeler transform of
a Christoffel word, see \cite{MR1976388,MR2197281} or
more recently \cite[Theorem 15.2.4]{MR3887697}.

Small local changes from a factor to the next (in radix order) can also be seen
in the language of a balanced sequence.
For example, we list below the factors of length 8 in the Fibonacci word
as well as the greatest (for the radix order) factor of length 7
and the smallest factor of length 9:
\[
\begin{array}{lllll}
                              & & \texttt{\underline{1}010010}  &=& \underline{\b}m\\
\texttt{0010\underline{01}01} &=& \texttt{0010\underline{01}01} &=& \a m\b\\
\texttt{0\underline{01}01001} &=& \texttt{0\underline{01}01001} & & \\
\texttt{010010\underline{01}} &=& \texttt{010010\underline{01}} & & \\
\texttt{010\underline{01}010} &=& \texttt{010\underline{01}010} &=& \widetilde{u}\a\b v\\
\texttt{\underline{01}010010} &=& \texttt{\underline{01}010010} &=& \widetilde{u}\b\a v\\
\texttt{10010\underline{01}0} &=& \texttt{10010\underline{01}0} & & \\
\texttt{10\underline{01}0100} &=& \texttt{10\underline{01}0100} & & \\
\texttt{10100100}             &=& \texttt{1010010\underline{0}} &=& \b m\underline{\a}\\
                              & & \texttt{\underline{1}0100101} &=& \underline{\b}m\b\\
                              & & \texttt{001001010}            &=& \a m\b\a
\end{array}
\]
In the left column, the 8 cyclic conjugates of the Christoffel word
\texttt{00100101} are listed in lexicographic order. 
It illustrates what happens in the context of the Burrows-Wheeler transform: in
each word (except the last), the factor $\a\b$ that is changed into a $\b\a$ is
underlined.
In the middle column, the 9 factors of length $8$ in the language of the Fibonacci word.
There is also the lexicographically largest factor of length 7 and the 
lexicographically smallest factor of length 9. As for the conjugates of a
Christoffel word, we observe that at most two letters change from a word
to the next.
The type of changes summarized in the right column
can be verified on the factors of the Fibonacci
word of length up to 9 ordered in radix order
in \Cref{T:fibo_factor} in the appendix.

This observation proved in \Cref{lem:cyclic-ordering-sturmian-factors}
is a key point in the proof of \Cref{thm:main-result-intro}.
More precisely, we observe that the small local changes are of the following forms:
\begin{equation}\label{eq:local-changes}
\widetilde{u}\a\b v
\mapsto
\widetilde{u}\b\a v,\qquad
w\a
\mapsto
w\b,\qquad
\b w
\mapsto
\a w\a,\qquad
\b w
\mapsto
\a w\b
\end{equation}
where $u,v,w\in\{\a,\b\}^*$ and $u$ is a prefix of $v$ or vice versa.

% sage: w = Word('00100101')
% sage: sorted(w.conjugates())

\section{Increasing over small local changes}
\label{sec:increasing-small-changes}

In this section, we prove that the map
$w\mapsto\mu_q(w)_{12}$ is increasing over the small local changes listed in
Equation~\eqref{eq:local-changes}.
Recall that we use the partial order $\prec$ on $\Z[q]$ is defined as
\[
    f \prec g
    \quad
    \text{ if and only if }
    \quad
    f \neq g
    \text{ and }
    g - f \in \Z_{\geq0}[q].
\]
More precisely, we prove the following two propositions.

\begin{proposition}\label{prop:bw-awa}
    For every $w\in\{\a,\b\}^*$,
    \[
        \mu_q(w\a)_{12}\prec \mu_q(w\b)_{12} 
        \quad
        \text{ and }
        \quad
        \mu_q(\b w)_{12}\prec 
        \mu_q(\a w\a)_{12}\prec
        \mu_q(\a w\b)_{12}.
    \]
\end{proposition}

\begin{proposition}\label{prop:growing-on-flips-ab-ba}
    Let $u,v\in\{\a,\b\}^*$ such that $u$ is a prefix of $v$ or vice versa. Then
    \[
          \mu_q(\widetilde{u} \a\b v)_{12}\prec
          \mu_q(\widetilde{u} \b\a v)_{12}.
    \]
\end{proposition}

The proof of each proposition is preceded by a lemma.
The proof of the two lemmas is the $q$-analog of the proofs made in
\cite{lapointe_these_2020,LR2021} over the integer entries.
In particular, the next lemma extends Lemma 2 from \cite{LR2021}.

\begin{lemma}\label{lem:two-inequalities}
    Let $w\in\{\a,\b\}^*$ and
        polynomials $m,n,o,p\in\Z[q]$ 
    such that 
    $\mu_q(w)= \left(\begin{smallmatrix} m &n\\o&p \end{smallmatrix}\right)$.
        Then $m$, $p$ and
    \begin{align}
        &qm - q^2 n + o ,                            \label{eq:intermediaire}\\
        &(q + q^2 )m -(q^2 + q^3 + q^4)n + o - qp,  \label{eq:proposition}
    \end{align}
    are nonzero polynomials with nonnegative coefficients.
    Moreover, $o$ and $n$
    are nonzero polynomials with nonnegative coefficients
    except if $w$ is empty in which case $o=n=0$.
\end{lemma}

\begin{proof}
    The proof is done by induction on the length of $w$.
    If $w$ is the empty word, then $m=p = 1$ and $n=o = 0$.
    Therefore,
    \begin{align*}
        &qm - q^2 n + o = q \in \Z_{\geq0}[q] \setminus\{0\},\\
        &(q + q^2 )m -(q^2 + q^3 + q^4)n + o - qp = ( q + q^2) - q = q^2 \in \Z_{\geq0}[q]\setminus\{0\} .
    \end{align*}

    Let $w\in\{\a,\b\}^*$ such that 
    $\mu_q(w)= \left(\begin{smallmatrix} m &n\\o&p \end{smallmatrix}\right)$
        for some polynomials $m,n,o,p\in\Z_{\geq0}[q]$.
    Assume by induction that $m$, $p$,
        \eqref{eq:intermediaire}
        and
        \eqref{eq:proposition}
    are nonzero polynomials with nonnegative coefficients.

    Let $w'\in\{\a,\b\}^*$ be a nonempty word $w' = w\a$ or $w' = w\b$.
    We have
    $\mu_q(w')= \left(\begin{smallmatrix} m' &n'\\o'&p' \end{smallmatrix}\right)$ for some polynomials $m',n',o',p'\in\Z[q]$.
    If $w' = w\a$, then
    \begin{align}
        \mu(w') = \begin{pmatrix} m' & n' \\ o' & p' \end{pmatrix}
        = \begin{pmatrix} m & n \\ o & p \end{pmatrix} \begin{pmatrix} q + q^2  & 1 \\ q & 1 \end{pmatrix} =
    \left(\begin{array}{rr}
{\left(q + q^{2}\right)} m + q n  & m + n \\
{\left(q + q^{2}\right)} o + q p  & o + p
\end{array}\right). \label{eq:produit_wa}
    \end{align}
    We observe that $m'$, $n'$, $o'$ and $p'$
    are nonzero polynomials with nonnegative coefficients.
    Also, we have 
    \begin{align*}
        q m' - q^{2} n' + o' &= q \left({\left(q + q^{2}\right)} m + q n \right) - q^{2} \left( m + n \right)  + \left( {\left(q + q^{2}\right)} o +  q p \right) \\
                             &= q {\left( q^{2} m +  (q + 1)o + p\right)}  \in \Z_{\geq0}[q] \setminus\{0\},
    \end{align*}
    since $m,p \in \Z_{\geq0}[q]\setminus\{0\}$ and $o \in \Z_{\geq0}[q]$.
    Moreover, using the induction hypothesis, we have
    \begin{align*}
        (q + q^2 )m' -(q^2 + q^3 + q^4)n' + o' - qp' &= (q + q^2 )\left({\left(q + q^{2}\right)} m  + q n\right)  -(q^2 + q^3 + q^4)(m+n) \\&\quad+ \left({\left(q + q^{2}\right)} o + q p \right)- q (o +p )\\
                                                     &= q^{2}{\left( q m - q^{2} n + o\right)} \in \Z_{\geq0}[q]\setminus\{0\}
    \end{align*}
    by Equation \eqref{eq:intermediaire}.

    If $w' = w\b$, then
    \begin{align*}
        \mu(w') &= 
        \begin{pmatrix} m' & n' \\ o' & p' \end{pmatrix} 
       =\begin{pmatrix} m & n \\ o & p \end{pmatrix} 
        \begin{pmatrix} q+2q^2 + q^3+q^4 & 1 + q \\ q + q^2 & 1 \end{pmatrix} \\&=
    \left(\begin{array}{rr}
{\left(q+2q^2 + q^3+q^4\right)} m + {\left( q + q^{2} \right)} n &  {\left(1 + q \right)}m + n \\
{\left(q+2q^2 + q^3+q^4\right)} o + {\left(q + q^{2}\right)} p &  {\left(1 + q \right)}o + p \label{eq:produit_wb}
\end{array}\right)
    \end{align*}
    We observe that $m'$, $n'$, $o'$ and $p'$
    are nonzero polynomials with nonnegative coefficients.
    Also we have that both 
    \begin{align*}
        q m' - q^{2} n' + o' &= q \left( {\left(q+2q^2 + q^3+q^4\right)} m + {\left( q + q^{2} \right)} n \right) - q^{2} \left( {\left(1 + q \right)}m + n \right) \\ &\quad+ \left( {\left(q+2q^2 + q^3+q^4\right)} o + {\left(q + q^{2}\right)} p \right)  \\ 
                             &= q\left( (q^2+ q^3+q^4)m + q^2n + (1 + 2q + q^2 + q^3)o + (1 + q)p\right)    
    \end{align*}
    and
    \begin{align*}
        (q + q^2 )m' -(q^2 + q^3 + q^4)n' + o' - qp' &= (q + q^2 ) \left( {\left(q+2q^2 + q^3+q^4\right)} m + {\left( q + q^{2} \right)} n \right)  \\ &\quad-(q^2 + q^3 + q^4) \left( {\left(1 + q \right)}m + n \right) \\&\quad+ \left({\left(q+2q^2 + q^3+q^4\right)} o + {\left(q + q^{2}\right)} p \right) \\ &\quad- q \left( {\left(1 + q \right)}o + p \right) \\ 
                                                     &= q^2((q + q^2+q^3+q^4)m + qn + (1 + q + q^2)o + p)    
    \end{align*}
    belong to $\Z_{\geq0}[q]\setminus\{0\}$ since $m, p \in \Z_{\geq0}[q]\setminus\{0\}$ and $n,o \in \Z_{\geq0}[q]$.
\end{proof}

\begin{proof}[Proof of \Cref{prop:bw-awa}]
    Let $w\in\{\a,\b\}^*$ such that $\mu_q(w)=
            \left(\begin{smallmatrix} m &n\\o&p \end{smallmatrix}\right)$
                where $m,n,o,p\in\Z_{\geq0}[q]$
    from \Cref{lem:two-inequalities}.
                Firstly, we have
    \begin{align*}
        \mu_q(w\b) - \mu_q(w\a)
        &= \mu_q(w) \left[\mu_q(\b) - \mu_q(\a)\right]\\
        &= \left(\begin{array}{cc} m &n\\o&p \end{array}\right)
            \left[\left(\begin{array}{cc} q+2q^2 + q^3 + q^4 & 1 +q\\q + q^2 & 1 \end{array}\right)
                -\left(\begin{array}{cc} q+q^2 &1\\q&1 \end{array}\right)\right]\\
            &= \left(\begin{array}{cc} m &n\\o&p \end{array}\right)
            \left(\begin{array}{cc} q^2 + q^3 + q^4 & q\\q^2 & 0 \end{array}\right)
    \end{align*}
    We compute the entry $(1,2)$ of the above matrix and we obtain
    \begin{align*}
        \mu_q(w\b)_{12} - \mu_q(w\a)_{12} = m q \in \Z_{\geq0}[q]\setminus\{0\}
    \end{align*}
    since $m \in \Z_{\geq0}[q]\setminus\{0\}$.

                Secondly, we have
    \begin{align*}
        \mu_q(\a w\a)
        - \mu_q(\b w)
        &=\mu_q(\a)
            \left(\begin{smallmatrix} m &n\\o&p \end{smallmatrix}\right)
                \mu_q(\a)
                -\mu_q(\b)
            \left(\begin{smallmatrix} m &n\\o&p \end{smallmatrix}\right) \\
        &=\left(\begin{smallmatrix} q+q^2 &1\\q&1 \end{smallmatrix}\right)
            \left(\begin{smallmatrix} m &n\\o&p \end{smallmatrix}\right)
            \left(\begin{smallmatrix} q+q^2 &1\\q&1 \end{smallmatrix}\right)    
                -\left(\begin{smallmatrix} q+2q^2 + q^3 + q^4 & 1 +q\\q + q^2 & 1 \end{smallmatrix}\right)
            \left(\begin{smallmatrix} m &n\\o&p \end{smallmatrix}\right).
    \end{align*}
    We compute the entry $(1,2)$ of the above matrix and
    from \Cref{lem:two-inequalities}, we obtain
    \begin{align*}
        \mu_q(\a w\a)_{12}
        - \mu_q(\b w)_{12}
        &={\left(q + q^{2} \right)} m  + {\left(q + q^{2} \right)} n + o + p - {\left(q + 2 \, q^2 + q^{3} + q^{4}\right)} n-  {\left(1 + q \right)} p\\
        &= (q + q^2  )m -(q^2 + q^3 + q^4)n + o - qp \in \Z_{\geq0}[q]\setminus\{0\}.
    \end{align*}
    From the first part of the proof, we also
    have $\mu_q(\a w\a)_{12} - \mu_q(\a w\b)_{12}$ is a nonzero polynomial with nonnegative coefficients.
\end{proof}

Let 
\[
    D_q = 
    \mu_q(\b\a) - \mu_q(\a\b) = 
\left(\begin{array}{rr}
    0 & q + q^{4} \\
    -q^{2} -q^{5}  & 0
\end{array}\right).
\]
The matrix $D_q$ represent the flip between two consecutive factors in the language of a balanced language. Hence, properties of this matrix are used to prove our main result.
In the next lemma, we use the following notation. If $u\in\{\a,\b\}^*$ and $a\in\{\a,\b\}$,
then $|u|_a$ is the number of occurrences of the letter $a$ in $u$.

\begin{lemma}\label{lem:ubau-uabu}
    Let  $u\in\{\a,\b\}^*$. Then
    \[
          \mu_q(\widetilde{u} \b\a u)
        - \mu_q(\widetilde{u} \a\b u)
        = q^n D_q
        =\det(\mu_q(u)) D_q
    \]
    where $n=2|u|_\a+4|u|_\b$.
\end{lemma}

\begin{proof}
It follows from the following two equalities
\begin{align*}
    \mu_q(\a) D_q \mu_q(\a) &= 
\left(\begin{array}{rr}
    q+q^2 & 1\\
    q  & 1
\end{array}\right)
\left(\begin{array}{rr}
    0 & q + q^{4} \\
    -q^{2} - q^{5} & 0
\end{array}\right)
\left(\begin{array}{rr}
    q+q^2 & 1\\
    q  & 1
\end{array}\right) \\
&= 
\left(\begin{array}{rr}
    0 & q^{3} + q^{6} \\
    -q^{4} - q^{7} & 0
\end{array}\right)
= q^2 D_q
    = \det(\mu_q(\a)) D_q
\end{align*}
and
\begin{align*}
    \mu_q(\b) D_q \mu_q(\b) &=
\left(\begin{array}{rr}
    q+2q^2 +q^3+q^4 & 1+q\\
    q + q^2  & 1
\end{array}\right)
\left(\begin{array}{rr}
    0 & q + q^{4} \\
    -q^{2} - q^{5} & 0
\end{array}\right)
\left(\begin{array}{rr}
    q+2q^2 +q^3+q^4 & 1+q\\
    q + q^2  & 1
\end{array}\right) \\
                            &=
\left(\begin{array}{rr}
    0 & q^{5} + q^{8} \\
    -q^{6} - q^{9} & 0
\end{array}\right)
= q^4 D_q
    = \det(\mu_q(\b)) D_q.
\end{align*}
\end{proof}

\begin{proof}[Proof of \Cref{prop:growing-on-flips-ab-ba}]
    There are three cases: 
    $u=v$,
    $v$ is longer than $u$ or 
    $u$ is longer than $v$.
    First assume $u=v$.
    We use the identity $\mu_q(\widetilde{u})D_q\mu_q(u)=q^n D_q$ for some $n>0$
    from Lemma~\ref{lem:ubau-uabu}.
    We have
    \begin{align*}
        \mu_q(\widetilde{u}\b\a v)_{12}-\mu_q(\widetilde{u}\a\b v)_{12}
        &= \left(\begin{smallmatrix} 1 & 0 \end{smallmatrix}\right) 
            \mu_q(\widetilde{u}) [\mu_q(\b\a)-\mu_q(\a\b)]\mu_q(u) 
            \left(\begin{smallmatrix} 0 \\ 1 \end{smallmatrix}\right)\\
        &= \left(\begin{smallmatrix} 1 & 0 \end{smallmatrix}\right) 
            \cdot q^n D_q \cdot
            \left(\begin{smallmatrix} 0 \\ 1 \end{smallmatrix}\right)
                = q^n \left(q^4+q\right) \in \Z_{\geq0}[q]\setminus\{0\}.
    \end{align*}
    Assume $|v|>|u|$ and let $s\in\{\a,\b\}^+$ such that $v=us$.
    We compute
    \begin{align*}
        \mu_q(\widetilde{u}\b\a v)_{12}-\mu_q(\widetilde{u}\a\b v)_{12}
        &= \left(\begin{smallmatrix} 1 & 0 \end{smallmatrix}\right) 
            \mu_q(\widetilde{u})[\mu_q(\b\a)-\mu_q(\a\b)]\mu_q(u)\mu_q(s)
            \left(\begin{smallmatrix} 0 \\ 1 \end{smallmatrix}\right)\\
        &= \left(\begin{smallmatrix} 1 & 0 \end{smallmatrix}\right) 
            \cdot q^n D_q\cdot \mu_q(s)
            \left(\begin{smallmatrix} 0 \\ 1 \end{smallmatrix}\right)\\
        &= q^n \left(\begin{smallmatrix} 0 & q^4+q \end{smallmatrix}\right) 
            \mu_q(s)
            \left(\begin{smallmatrix} 0 \\ 1 \end{smallmatrix}\right) \in \Z_{\geq0}[q]\setminus\{0\},
    \end{align*}
    since $s$ is non-empty and
    from \Cref{lem:two-inequalities}
    the entries of $\mu_q(s)$ are nonzero
    polynomials with nonnegative coefficients.

    Assume $|u|>|v|$  and let $s\in\{\a,\b\}^+$ such that $u=vs$.
    We obtain
    \begin{align*}
        \mu_q(\widetilde{u}\b\a v)_{12}-\mu_q(\widetilde{u}\a\b v)_{12}
        &= \left(\begin{smallmatrix} 1 & 0 \end{smallmatrix}\right) 
            \mu_q(\widetilde{s})\mu_q(\widetilde{v})\left[\mu_q(\b\a)-\mu_q(\a\b)\right]\mu_q(v)
            \left(\begin{smallmatrix} 0 \\ 1 \end{smallmatrix}\right)\\
        &= \left(\begin{smallmatrix} 1 & 0 \end{smallmatrix}\right) 
            \mu_q(\widetilde{s})\cdot q^nD_q\cdot
            \left(\begin{smallmatrix} 0 \\ 1 \end{smallmatrix}\right)\\
        &= q^n\left(\begin{smallmatrix} 1 & 0 \end{smallmatrix}\right) 
            \mu_q(\widetilde{s})
            \left(\begin{smallmatrix} q^4+q \\ 0 \end{smallmatrix}\right)\in \Z_{\geq0}[q]\setminus\{0\},
    \end{align*}
    since $s$ is non-empty and
    from \Cref{lem:two-inequalities}
    the entries of $\mu_q(\widetilde{s})$ are nonzero polynomials with
    nonnegative coefficients.
\end{proof}

\section{Proof of \Cref{thm:main-result-intro}}
\label{sec:proof-main-results}

The following lemma can be seen as a extension to biinfinite balanced sequences
of complexity $n+1$ of Theorem 15.2.4 from \cite{MR3887697}
stated for the $n$ conjugates of a Christoffel word of length $n$.

\begin{lemma}\label{lem:cyclic-ordering-sturmian-factors}
Let $s\in\{\a,\b\}^\Z$ be a balanced sequence
having at least one factorization $s=\widetilde{p}xyp$
        where $\{x,y\}=\{\a,\b\}$.
        Let $n\geq1$ and $u_0,\dots,u_{n}$ be the $n+1$ factors of length $n$
        of $s$
    such that 
    \[
    u_0<_{lex}\dots<_{lex}u_{n}.
    \]
    If $w$ is the prefix of length $n-1$ of $p$, we have
    \begin{itemize}
        \item $u_0=\a w$ and $u_{n}=\b w$,
        \item there exists $i\in\{0,\dots,n-1\}$ such that 
            $u_i=\widetilde{w}\a$ and $u_{i+1}=\widetilde{w}\b$,
        \item for all $j\in\{0,\dots,n-1\}\setminus\{i\}$,
            there exist prefixes $x,y$ of $w$
            such that $u_j=\widetilde{x}\a\b y$ and $u_{j+1}=\widetilde{x}\b\a y$.
    \end{itemize}
\end{lemma}

\begin{proof}
    Let $f:\Lcal_n(s)\to\Lcal_n(s)$ be the map such that 
    $f(u)$ is the factor appearing in $\widetilde{w}\b\a w$
    at the same position as the occurrence of $u$ in $\widetilde{w}\a\b w$.
    From \Cref{cor:hamiltonian-sturmian}, $f$ is a bijection.

    From the definition of $f$, if $u\neq \b w$ then $u<_{lex}f(u)$.
    Thus $f$ is a cyclic permutation (if $f$ had at least 2 cycles, there would
    exist two distinct words $u$ such that $u\not<_{lex}f(u)$).
    Thus there exists a minimal word $u_0$ for the lexicographic order such that
    \[
        u_0
        <_{lex}f(u_0)
        <_{lex}f^2(u_0)
        <_{lex}\dots
        <_{lex}f^n(u_0)
        \not<_{lex}f^{n+1}(u_0)=u_0
    \]
    It also implies that the maximal factor for the lexicographic order is
    $f^n(u_0)=\b w$
    and the minimal one is $u_0=f(f^n(u_0))=f(\b w)=\a w$.
    This ends the proof if we let $u_{i}=f^i(u_0)$ for all $i\in\{1,\dots,n\}$.
\end{proof}

We may now show that the $q$-analog of the Markoff injectivity conjecture holds over the language
of a balanced sequence satisfying property $(M_3)$ or $(M_4)$.

\begin{proposition}\label{prop:increasing-on-a-sturmian}
Let $s\in\{\a,\b\}^\Z$ be a balanced sequence
having at least one factorization $s=\widetilde{p}xyp$
        where $\{x,y\}=\{\a,\b\}$.
    Let $u,v\in\Lcal(s)$ be two factors in the language
    of $s$.
    If $u<_{radix}v$, then 
    $\mu_q(u)_{12}\prec\mu_q(v)_{12}$.
\end{proposition}

\begin{proof}
    Let $(x_i)_{0\leq i\leq n}$ be a maximal chain
    for the radix order such that
    \[
        u=x_0
        <_{radix} x_1
        <_{radix} \cdots
        <_{radix} x_n = v.
    \]
    Let $i\in\{0,\dots,n-1\}$.
    Since the chain is maximal,
    we have $|x_{i+1}|=|x_i|$ or $|x_{i+1}|=|x_i|+1$.
    First assume that $|x_i|=|x_{i+1}|$.
    Since $x_i<_{radix}x_{i+1}$, we have $x_i<_{lex}x_{i+1}$.
    From
\Cref{lem:cyclic-ordering-sturmian-factors}
and
\Cref{prop:growing-on-flips-ab-ba}, we conclude that
    $\mu_q(x_i)_{12}\prec\mu_q(x_{i+1})_{12}$.

    Now assume that $|x_{i+1}|=|x_i|+1$.
    Since the chain is maximal, then $x_i$ is the lexicographically maximal factor of its length
    and $x_{i+1}$ is the lexicographically minimal factor of its length.
    From \Cref{lem:cyclic-ordering-sturmian-factors},
    the prefix $w$ of length $|x_i|-1$ of $p$ is such that
        $x_i = \b w$
        and
        $x_{i+1} \in\{\a w\a,\a w\b\}$.
    From \Cref{prop:bw-awa}, we obtain that
    $\mu_q(x_i)_{12}\prec\mu_q(x_{i+1})_{12}$.
\end{proof}

We may now prove the main result and its corollaries.

\begin{proof}[Proof of \Cref{thm:main-result-intro}]
    Let $s\in\{\a,\b\}^\Z$ be a balanced sequence over the alphabet
    $\Sigma=\{\a,\b\}$.
    The sequence $s$ is in one of the four
    cases: $(M_1)$, $(M_2)$, $(M_3)$ or $(M_4)$.

    If $s$ satisfies case $(M_1)$, then from \Cref{thm:balanced-4-cases}, $s$ is a purely periodic sequence 
        $\leftidx{^\infty}w^\infty$ for some Christoffel word $w$.
        From \Cref{thm:balanced-4-cases},
    there exists an eventually periodic sequence $s'$
    satisfying $(MH_4)$ and $(M_4)$
    such that $\Lcal(s)\subset\Lcal(s')$.
    If $s$ satisfies case $(M_2)$,
    then there exists a characteristic Sturmian sequence $s'$
    of the same slope
    satisfying case $(M_3)$ such that $\Lcal(s)=\Lcal(s')$.
    If $s$ already satisfies case $(M_3)$ or $(M_4)$, then let $s'=s$.

    In summary, $s'$ is a balanced sequence satisfying
    property $(M_3)$ or $(M_4)$ such that $\Lcal(s)\subset\Lcal(s')$.
Thus $s'\in\{\a,\b\}^\Z$ is a balanced sequence
having at least one factorization $s'=\widetilde{p}xyp$
        where $\{x,y\}=\{\a,\b\}$.
    From \Cref{prop:increasing-on-a-sturmian},
    if $u,v\in\Lcal(s')$ 
    such that $u<_{radix}v$, then 
    $\mu_q(u)_{12}\prec\mu_q(v)_{12}$.
\end{proof}

\begin{proof}[Proof of \Cref{cor:main-corollary-intro-injective}]
    Let $s\in\{\a,\b\}^\Z$ be a balanced sequence.
    Let $u,v\in\Lcal(s)$ such that $u\neq v$. Without loss of generality,
    we may assume that $u<_{radix}v$.
    From \Cref{thm:main-result-intro}, $\mu_q(u)_{12}\prec\mu_q(v)_{12}$.
    In particular, $\mu_q(u)_{12}\neq\mu_q(v)_{12}$.
    Therefore, the map $u\mapsto\mu_q(u)_{12}$ is injective over the language $\Lcal(s)$.
\end{proof}

\begin{proof}[Proof of \Cref{cor:main-corollary-intro-q>0}]
    Let $u,v\in\Lcal(s)$ be two factors in the language of
    a balanced sequence
    $s\in\{\a,\b\}^\Z$
    such that $u<_{radix}v$.
    From \Cref{thm:main-result-intro},
    $\mu_q(u)_{12}\prec\mu_q(v)_{12}$.
    Thus $\mu_q(v)_{12} - \mu_q(u)_{12}$
    is a nonzero polynomial with nonnegative coefficients.
    In particular,
    for every $\gamma>0$,
    $(\mu_q(v)_{12} - \mu_q(u)_{12})|_{q=\gamma} > 0$
    or equivalently
    $\mu_q(u)_{12}|_{q=\gamma} < \mu_q(v)_{12}|_{q=\gamma}$.
    In particular,
    $\mu_q(u)_{12}|_{q=\gamma} \neq \mu_q(v)_{12}|_{q=\gamma}$.
    Thus the map 
    $\{\a,\b\}^*\to\R$ defined by
    $w\mapsto\mu_q(w)_{12}|_{q=\gamma}$ is strictly increasing 
    and injective over $\Lcal(s)$ for every $\gamma>0$.
\end{proof}

\section{Conclusion}\label{sec:conclusion}

As we have shown in \Cref{thm:main-result-intro},
the map $w\mapsto\mu_q(w)_{12}$ is increasing over the
language of a balanced sequence. 
But this is not true for the language of all
balanced sequences. Thus the Markoff injectivity conjecture can not be
extended to the injectivity of the map $w\mapsto\mu_q(w)_{12}|_{q=\gamma}$ for all $\gamma>0$. 
We provide few counterexamples below.

%sage: baa = word_to_markoff_number('baa', symbolic='q')
%sage: abb = word_to_markoff_number('abb', symbolic='q')
%sage: latex(abb)
%sage: latex(baa)
%sage: latex(abb-baa)
%sage: plot((abb-baa).univariate_polynomial())

Observe that
$\a\b\b <_{radix} \b\a\a$ but
\begin{align*}
    \mu_q(\b\a\a)_{12}&=
1 + 3 q + 4 q^{2} + 4 q^{3} + 4 q^{4} + 2 q^{5} + q^{6},\\
    \mu_q(\a\b\b)_{12}&=
1 + 2 q + 5 q^{2} + 6 q^{3} + 6 q^{4} + 5 q^{5} + 3 q^{6} + q^{7},
\end{align*}
and
    \[
    \mu_q(\b\a\a)_{12}-
    \mu_q(\a\b\b)_{12}
    =
q - q^{2} - 2 q^{3} - 2 q^{4} - 3 q^{5} - 2 q^{6} - q^{7}
\]
has negative coefficients.

Another example is
$\a\b\b\b\a\b <_{radix} \b\a\b\a\b\b$
but
\begin{align*}
    \mu_q(\b\a\b\a\b\b)_{12} - \mu_q(\a\b\b\b\a\b)_{12} = q + 3q^{2} + 7q^{3} + 12q^{4} + 17q^{5} + 20q^{6}+ 21q^{7}  + 19q^{8} \\ + 14q^{9} + 9q^{10}  + 4q^{11}+ q^{12} - q^{13} -q^{14}
\end{align*}
which has negative coefficients.

The words $\b\a\a$,
$\b\a\b\a\b\b$ and $\a\b\b\b\a\b$
are not Christoffel words, but even in the case of Christoffel words, there are counterexamples.
For example,
\begin{align*}
    \mu_q(\a\a\a\a\b)_{12} &= 1 + 4q + 8q^{2} + 13q^{3} + 16q^{4} + 17q^{5} + 14q^{6} + 10q^{7} + 5q^{8} + q^{9},\\
    \mu_q(\a\b\b\b)_{12} &=1 + 3q + 9q^{2} + 16q^{3} + 24q^{4} + 29q^{5} + 29q^{6} + 25q^{7} + 18q^{8} + 10q^{9} + 4q^{10} + q^{11},
\end{align*}
and their difference
\begin{align*}
    \mu_q(\a\a\a\a\b)_{12} -
    \mu_q(\a\b\b\b)_{12}
    = q- q^{2} - 3q^{3}  - 8q^{4} - 12q^{5} - 15q^{6}- 15q^{7} - 13q^{8} - 9q^{9} - 4q^{10} - q^{11}
\end{align*}
has negative coefficients.
Another counterexample is given by the Christoffel words $\a^{12}\b$ and
$\a\b^7$.

The above counterexamples show that 
\Cref{cor:main-corollary-intro-q>0} does not hold for $\gamma>0$ over the
language of all balanced sequences.
Thus a proof of Markoff Injectivity Conjecture needs to use the hypothesis that
polynomials are evaluated only at $q=1$, 
perhaps extending the approach used in \cite{MR2543340}.

% \todo[inline]{add explanations below}

% For Markoff injectivity conjecture,
% $q > 1$ 
% \begin{align*}
%     \mu_q(a^{12}b) - \mu_q(ab^7) = -q^{27} - 8q^{26} - 39q^{25} - 133q^{24 }- 348q^{23} - 743q^{22} - 1348q^{21} - 2122q^{20} \\- 2928q^{19} - 3550q^{18} - 3745q^{17} - 3334q^{16} - 2283q^{15} - 743q^{14} + 980q^{13} \\ + 2520q^{12} + 3566q^{11} + 3961q^{10} + 3735q^{9} + 3070q^{8} + 2215q^{7 } \\+ 1398q^{6} + 761q^{5} + 347q^{4} + 127q^{3} + 33q^{2} + 5q
% \end{align*}

% $0< q < 1$
% \begin{align*}
%     \mu_q(abbb) - \mu_q(aaaab) = q^{11} + 4q^{10} + 9q^{9} + 13q^{8} + 15q^{7} + 15q^{6} + 12q^{5} + 8q^{4} + 3q^{3} + q^{2} - q
% \end{align*}

% \begin{align*}
%     \mu_q(abbb) &=1 + 3q + 9q^{2} + 16q^{3} + 24q^{4} + 29q^{5} + 29q^{6} + 25q^{7} + 18q^{8} + 10q^{9} + 4q^{10} + q^{11} \\ 
%     \mu_q(aaaab) &= 1 + 4q + 8q^{2} + 13q^{3} + 16q^{4} + 17q^{5} + 14q^{6} + 10q^{7} + 5q^{8} + q^{9}
% \end{align*}

Finally, we observe that the map $w\mapsto\mu_q(w)_{12}$ from $\{\a,\b\}^*$ to
polynomials in the indeterminate $q$ is not injective, as for example:
\begin{align*}
    \mu_q(\a\a\a\b\b\b)_{12}&=
    1 + 5 q + 16 q^{2} + 38 q^{3} + 70 q^{4} + 109 q^{5} + 145 q^{6} + 168 q^{7} + 171 q^{8} \\
    &\qquad+ 152 q^{9} + 118 q^{10} + 79 q^{11} + 44 q^{12} + 19 q^{13} + 6 q^{14} + q^{15}\\
    &=\mu_q(\a\b\b\a\a\b)_{12}.
\end{align*}

%\newpage
\section{Appendix: values of $\mu_q(w)_{12}$ over the language of the Fibonacci word}

The following table gathers the values of $\mu(w)_{12}$ and
$\mu_q(w)_{12}$ over the language of the Fibonacci word for factors of length
up to 9 sorted in radix order. We observe that the coefficients of the
polynomials are increasing from one row to the next.
    The graph of the 55 polynomials listed in the table 
    is shown in \Cref{fig:log_polynomials}. The difference between two polynomials of
    the same degree can't be seen in the figure as it is relatively very small.

% sage: s = SturmianSubshift(1/golden_ratio^2)
% sage: latex(s.table_Markoff_values(10))

\begin{longtable}{l|l|p{13cm}}
$w$ & $\mu(w)_{12}$ & $\mu_q(w)_{12}$ \\ \hline
$\varepsilon$ & $0$ & $0$ \\
$\color{orange}\underline{\a}$ & $\color{orange}\underline{1}$ & $1$ \\
$\color{orange}\underline{\b}$ & $\color{orange}\underline{2}$ & $1 + q$ \\
$\a\a$ & $3$ & $1 + q + q^{2}$ \\
$\color{orange}\underline{\a\b}$ & $\color{orange}\underline{5}$ & $1 + q + 2 q^{2} + q^{3}$ \\
$\b\a$ & $7$ & $1 + 2 q + 2 q^{2} + q^{3} + q^{4}$ \\
$\color{orange}\underline{\a\a\b}$ & $\color{orange}\underline{13}$ & $1 + 2 q + 3 q^{2} + 3 q^{3} + 3 q^{4} + q^{5}$ \\
$\a\b\a$ & $17$ & $1 + 2 q + 4 q^{2} + 4 q^{3} + 3 q^{4} + 2 q^{5} + q^{6}$ \\
$\b\a\a$ & $19$ & $1 + 3 q + 4 q^{2} + 4 q^{3} + 4 q^{4} + 2 q^{5} + q^{6}$ \\
$\b\a\b$ & $31$ & $1 + 3 q + 5 q^{2} + 6 q^{3} + 7 q^{4} + 5 q^{5} + 3 q^{6} + q^{7}$ \\
$\a\a\b\a$ & $44$ & $1 + 3 q + 6 q^{2} + 8 q^{3} + 9 q^{4} + 8 q^{5} + 5 q^{6} + 3 q^{7} + q^{8}$ \\
$\a\b\a\a$ & $46$ & $1 + 3 q + 6 q^{2} + 9 q^{3} + 9 q^{4} + 8 q^{5} + 6 q^{6} + 3 q^{7} + q^{8}$ \\
$\a\b\a\b$ & $75$ & $1 + 3 q + 7 q^{2} + 11 q^{3} + 14 q^{4} + 14 q^{5} + 12 q^{6} + 8 q^{7} + 4 q^{8} + q^{9}$ \\
$\b\a\a\b$ & $81$ & $1 + 4 q + 8 q^{2} + 12 q^{3} + 15 q^{4} + 15 q^{5} + 13 q^{6} + 8 q^{7} + 4 q^{8} + q^{9}$ \\
$\b\a\b\a$ & $105$ & $1 + 4 q + 9 q^{2} + 14 q^{3} + 18 q^{4} + 19 q^{5} + 17 q^{6} + 12 q^{7} + 7 q^{8} + 3 q^{9} + q^{10}$ \\
$\a\a\b\a\a$ & $119$ & $1 + 4 q + 9 q^{2} + 15 q^{3} + 20 q^{4} + 22 q^{5} + 19 q^{6} + 15 q^{7} + 9 q^{8} + 4 q^{9} + q^{10}$ \\
$\color{orange}\underline{\a\a\b\a\b}$ & $\color{orange}\underline{194}$ & $1 + 4 q + 10 q^{2} + 18 q^{3} + 27 q^{4} + 33 q^{5} + 33 q^{6} + 29 q^{7} + 21 q^{8} + 12 q^{9} + 5 q^{10} + q^{11}$ \\
$\a\b\a\a\b$ & $196$ & $1 + 4 q + 10 q^{2} + 19 q^{3} + 27 q^{4} + 33 q^{5} + 34 q^{6} + 29 q^{7} + 21 q^{8} + 12 q^{9} + 5 q^{10} + q^{11}$ \\
$\a\b\a\b\a$ & $254$ & $1 + 4 q + 11 q^{2} + 21 q^{3} + 32 q^{4} + 40 q^{5} + 42 q^{6} + 39 q^{7} + 30 q^{8} + 19 q^{9} + 10 q^{10} + 4 q^{11} + q^{12}$ \\
$\b\a\a\b\a$ & $274$ & $1 + 5 q + 13 q^{2} + 24 q^{3} + 35 q^{4} + 43 q^{5} + 46 q^{6} + 41 q^{7} + 31 q^{8} + 20 q^{9} + 10 q^{10} + 4 q^{11} + q^{12}$ \\
$\b\a\b\a\a$ & $284$ & $1 + 5 q + 13 q^{2} + 24 q^{3} + 36 q^{4} + 45 q^{5} + 47 q^{6} + 43 q^{7} + 33 q^{8} + 21 q^{9} + 11 q^{10} + 4 q^{11} + q^{12}$ \\
$\a\a\b\a\a\b$ & $507$ & $1 + 5 q + 14 q^{2} + 29 q^{3} + 48 q^{4} + 67 q^{5} + 79 q^{6} + 81 q^{7} + 71 q^{8} + 54 q^{9} + 34 q^{10} + 17 q^{11} + 6 q^{12} + q^{13}$ \\
$\a\a\b\a\b\a$ & $657$ & $1 + 5 q + 15 q^{2} + 32 q^{3} + 55 q^{4} + 79 q^{5} + 96 q^{6} + 102 q^{7} + 94 q^{8} + 76 q^{9} + 52 q^{10} + 30 q^{11} + 14 q^{12} + 5 q^{13} + q^{14}$ \\
$\a\b\a\a\b\a$ & $663$ & $1 + 5 q + 15 q^{2} + 33 q^{3} + 56 q^{4} + 80 q^{5} + 97 q^{6} + 103 q^{7} + 95 q^{8} + 76 q^{9} + 52 q^{10} + 30 q^{11} + 14 q^{12} + 5 q^{13} + q^{14}$ \\
$\a\b\a\b\a\a$ & $687$ & $1 + 5 q + 15 q^{2} + 33 q^{3} + 57 q^{4} + 82 q^{5} + 100 q^{6} + 107 q^{7} + 99 q^{8} + 80 q^{9} + 55 q^{10} + 32 q^{11} + 15 q^{12} + 5 q^{13} + q^{14}$ \\
$\b\a\a\b\a\a$ & $741$ & $1 + 6 q + 18 q^{2} + 38 q^{3} + 64 q^{4} + 90 q^{5} + 109 q^{6} + 115 q^{7} + 105 q^{8} + 84 q^{9} + 57 q^{10} + 33 q^{11} + 15 q^{12} + 5 q^{13} + q^{14}$ \\
$\b\a\a\b\a\b$ & $1208$ & $1 + 6 q + 19 q^{2} + 43 q^{3} + 78 q^{4} + 119 q^{5} + 156 q^{6} + 178 q^{7} + 179 q^{8} + 158 q^{9} + 121 q^{10} + 80 q^{11} + 44 q^{12} + 19 q^{13} + 6 q^{14} + q^{15}$ \\
$\b\a\b\a\a\b$ & $1210$ & $1 + 6 q + 19 q^{2} + 43 q^{3} + 78 q^{4} + 119 q^{5} + 156 q^{6} + 179 q^{7} + 179 q^{8} + 158 q^{9} + 122 q^{10} + 80 q^{11} + 44 q^{12} + 19 q^{13} + 6 q^{14} + q^{15}$ \\
$\a\a\b\a\a\b\a$ & $1715$ & $1 + 6 q + 20 q^{2} + 48 q^{3} + 91 q^{4} + 145 q^{5} + 198 q^{6} + 237 q^{7} + 249 q^{8} + 233 q^{9} + 192 q^{10} + 138 q^{11} + 86 q^{12} + 45 q^{13} + 19 q^{14} + 6 q^{15} + q^{16}$ \\
$\a\a\b\a\b\a\a$ & $1777$ & $1 + 6 q + 20 q^{2} + 48 q^{3} + 92 q^{4} + 148 q^{5} + 203 q^{6} + 244 q^{7} + 259 q^{8} + 243 q^{9} + 201 q^{10} + 146 q^{11} + 91 q^{12} + 48 q^{13} + 20 q^{14} + 6 q^{15} + q^{16}$ \\
$\a\b\a\a\b\a\a$ & $1793$ & $1 + 6 q + 20 q^{2} + 49 q^{3} + 94 q^{4} + 150 q^{5} + 206 q^{6} + 247 q^{7} + 261 q^{8} + 245 q^{9} + 202 q^{10} + 146 q^{11} + 91 q^{12} + 48 q^{13} + 20 q^{14} + 6 q^{15} + q^{16}$ \\
$\a\b\a\a\b\a\b$ & $2923$ & $1 + 6 q + 21 q^{2} + 54 q^{3} + 110 q^{4} + 188 q^{5} + 276 q^{6} + 356 q^{7} + 405 q^{8} + 411 q^{9} + 371 q^{10} + 296 q^{11} + 207 q^{12} + 125 q^{13} + 63 q^{14} + 25 q^{15} + 7 q^{16} + q^{17}$ \\
$\a\b\a\b\a\a\b$ & $2927$ & $1 + 6 q + 21 q^{2} + 54 q^{3} + 110 q^{4} + 188 q^{5} + 276 q^{6} + 356 q^{7} + 406 q^{8} + 412 q^{9} + 371 q^{10} + 297 q^{11} + 208 q^{12} + 125 q^{13} + 63 q^{14} + 25 q^{15} + 7 q^{16} + q^{17}$ \\
$\b\a\a\b\a\a\b$ & $3157$ & $1 + 7 q + 25 q^{2} + 63 q^{3} + 126 q^{4} + 211 q^{5} + 306 q^{6} + 390 q^{7} + 439 q^{8} + 441 q^{9} + 394 q^{10} + 312 q^{11} + 216 q^{12} + 129 q^{13} + 64 q^{14} + 25 q^{15} + 7 q^{16} + q^{17}$ \\
$\b\a\a\b\a\b\a$ & $4091$ & $1 + 7 q + 26 q^{2} + 68 q^{3} + 140 q^{4} + 241 q^{5} + 358 q^{6} + 467 q^{7} + 542 q^{8} + 562 q^{9} + 521 q^{10} + 433 q^{11} + 319 q^{12} + 207 q^{13} + 116 q^{14} + 55 q^{15} + 21 q^{16} + 6 q^{17} + q^{18}$ \\
$\b\a\b\a\a\b\a$ & $4093$ & $1 + 7 q + 26 q^{2} + 68 q^{3} + 140 q^{4} + 241 q^{5} + 358 q^{6} + 468 q^{7} + 542 q^{8} + 562 q^{9} + 522 q^{10} + 433 q^{11} + 319 q^{12} + 207 q^{13} + 116 q^{14} + 55 q^{15} + 21 q^{16} + 6 q^{17} + q^{18}$ \\
$\color{orange}\underline{\a\a\b\a\a\b\a\b}$ & $\color{orange}\underline{7561}$ & $1 + 7 q + 27 q^{2} + 75 q^{3} + 166 q^{4} + 309 q^{5} + 496 q^{6} + 701 q^{7} + 881 q^{8} + 994 q^{9} + 1008 q^{10} + 920 q^{11} + 753 q^{12} + 548 q^{13} + 351 q^{14} + 194 q^{15} + 89 q^{16} + 32 q^{17} + 8 q^{18} + q^{19}$ \\
$\a\a\b\a\b\a\a\b$ & $7571$ & $1 + 7 q + 27 q^{2} + 75 q^{3} + 166 q^{4} + 309 q^{5} + 496 q^{6} + 701 q^{7} + 882 q^{8} + 995 q^{9} + 1010 q^{10} + 922 q^{11} + 754 q^{12} + 550 q^{13} + 352 q^{14} + 194 q^{15} + 89 q^{16} + 32 q^{17} + 8 q^{18} + q^{19}$ \\
$\a\b\a\a\b\a\a\b$ & $7639$ & $1 + 7 q + 27 q^{2} + 76 q^{3} + 169 q^{4} + 314 q^{5} + 504 q^{6} + 711 q^{7} + 893 q^{8} + 1006 q^{9} + 1018 q^{10} + 928 q^{11} + 758 q^{12} + 551 q^{13} + 352 q^{14} + 194 q^{15} + 89 q^{16} + 32 q^{17} + 8 q^{18} + q^{19}$ \\
$\a\b\a\a\b\a\b\a$ & $9899$ & $1 + 7 q + 28 q^{2} + 81 q^{3} + 185 q^{4} + 353 q^{5} + 579 q^{6} + 836 q^{7} + 1075 q^{8} + 1242 q^{9} + 1296 q^{10} + 1222 q^{11} + 1040 q^{12} + 797 q^{13} + 545 q^{14} + 329 q^{15} + 172 q^{16} + 76 q^{17} + 27 q^{18} + 7 q^{19} + q^{20}$ \\
$\a\b\a\b\a\a\b\a$ & $9901$ & $1 + 7 q + 28 q^{2} + 81 q^{3} + 185 q^{4} + 353 q^{5} + 579 q^{6} + 836 q^{7} + 1075 q^{8} + 1243 q^{9} + 1296 q^{10} + 1222 q^{11} + 1041 q^{12} + 797 q^{13} + 545 q^{14} + 329 q^{15} + 172 q^{16} + 76 q^{17} + 27 q^{18} + 7 q^{19} + q^{20}$ \\
$\b\a\a\b\a\a\b\a$ & $10679$ & $1 + 8 q + 33 q^{2} + 95 q^{3} + 214 q^{4} + 401 q^{5} + 649 q^{6} + 926 q^{7} + 1178 q^{8} + 1348 q^{9} + 1393 q^{10} + 1303 q^{11} + 1100 q^{12} + 836 q^{13} + 567 q^{14} + 339 q^{15} + 176 q^{16} + 77 q^{17} + 27 q^{18} + 7 q^{19} + q^{20}$ \\
$\b\a\a\b\a\b\a\a$ & $11065$ & $1 + 8 q + 33 q^{2} + 95 q^{3} + 215 q^{4} + 406 q^{5} + 661 q^{6} + 947 q^{7} + 1211 q^{8} + 1393 q^{9} + 1446 q^{10} + 1358 q^{11} + 1152 q^{12} + 879 q^{13} + 598 q^{14} + 359 q^{15} + 186 q^{16} + 81 q^{17} + 28 q^{18} + 7 q^{19} + q^{20}$ \\
$\b\a\b\a\a\b\a\a$ & $11069$ & $1 + 8 q + 33 q^{2} + 95 q^{3} + 215 q^{4} + 406 q^{5} + 661 q^{6} + 948 q^{7} + 1212 q^{8} + 1393 q^{9} + 1447 q^{10} + 1359 q^{11} + 1152 q^{12} + 879 q^{13} + 598 q^{14} + 359 q^{15} + 186 q^{16} + 81 q^{17} + 28 q^{18} + 7 q^{19} + q^{20}$ \\
$\b\a\b\a\a\b\a\b$ & $18045$ & $1 + 8 q + 34 q^{2} + 102 q^{3} + 242 q^{4} + 481 q^{5} + 826 q^{6} + 1251 q^{7} + 1692 q^{8} + 2063 q^{9} + 2278 q^{10} + 2284 q^{11} + 2078 q^{12} + 1712 q^{13} + 1270 q^{14} + 841 q^{15} + 490 q^{16} + 246 q^{17} + 103 q^{18} + 34 q^{19} + 8 q^{20} + q^{21}$ \\
$\a\a\b\a\a\b\a\b\a$ & $25606$ & $1 + 8 q + 35 q^{2} + 109 q^{3} + 268 q^{4} + 551 q^{5} + 977 q^{6} + 1527 q^{7} + 2132 q^{8} + 2686 q^{9} + 3071 q^{10} + 3198 q^{11} + 3037 q^{12} + 2626 q^{13} + 2063 q^{14} + 1464 q^{15} + 930 q^{16} + 522 q^{17} + 254 q^{18} + 104 q^{19} + 34 q^{20} + 8 q^{21} + q^{22}$ \\
$\a\a\b\a\b\a\a\b\a$ & $25610$ & $1 + 8 q + 35 q^{2} + 109 q^{3} + 268 q^{4} + 551 q^{5} + 977 q^{6} + 1527 q^{7} + 2132 q^{8} + 2686 q^{9} + 3072 q^{10} + 3199 q^{11} + 3037 q^{12} + 2627 q^{13} + 2064 q^{14} + 1464 q^{15} + 930 q^{16} + 522 q^{17} + 254 q^{18} + 104 q^{19} + 34 q^{20} + 8 q^{21} + q^{22}$ \\
$\a\b\a\a\b\a\a\b\a$ & $25840$ & $1 + 8 q + 35 q^{2} + 110 q^{3} + 272 q^{4} + 560 q^{5} + 993 q^{6} + 1550 q^{7} + 2162 q^{8} + 2720 q^{9} + 3105 q^{10} + 3228 q^{11} + 3060 q^{12} + 2642 q^{13} + 2072 q^{14} + 1468 q^{15} + 931 q^{16} + 522 q^{17} + 254 q^{18} + 104 q^{19} + 34 q^{20} + 8 q^{21} + q^{22}$ \\
$\a\b\a\a\b\a\b\a\a$ & $26774$ & $1 + 8 q + 35 q^{2} + 110 q^{3} + 273 q^{4} + 565 q^{5} + 1007 q^{6} + 1580 q^{7} + 2214 q^{8} + 2797 q^{9} + 3208 q^{10} + 3349 q^{11} + 3187 q^{12} + 2763 q^{13} + 2175 q^{14} + 1546 q^{15} + 983 q^{16} + 552 q^{17} + 268 q^{18} + 109 q^{19} + 35 q^{20} + 8 q^{21} + q^{22}$ \\
$\a\b\a\b\a\a\b\a\a$ & $26776$ & $1 + 8 q + 35 q^{2} + 110 q^{3} + 273 q^{4} + 565 q^{5} + 1007 q^{6} + 1580 q^{7} + 2214 q^{8} + 2798 q^{9} + 3208 q^{10} + 3349 q^{11} + 3188 q^{12} + 2763 q^{13} + 2175 q^{14} + 1546 q^{15} + 983 q^{16} + 552 q^{17} + 268 q^{18} + 109 q^{19} + 35 q^{20} + 8 q^{21} + q^{22}$ \\
$\a\b\a\b\a\a\b\a\b$ & $43651$ & $1 + 8 q + 36 q^{2} + 117 q^{3} + 302 q^{4} + 653 q^{5} + 1219 q^{6} + 2008 q^{7} + 2958 q^{8} + 3937 q^{9} + 4763 q^{10} + 5261 q^{11} + 5315 q^{12} + 4910 q^{13} + 4141 q^{14} + 3176 q^{15} + 2200 q^{16} + 1363 q^{17} + 744 q^{18} + 350 q^{19} + 137 q^{20} + 42 q^{21} + 9 q^{22} + q^{23}$ \\
$\b\a\a\b\a\a\b\a\b$ & $47081$ & $1 + 9 q + 42 q^{2} + 137 q^{3} + 351 q^{4} + 750 q^{5} + 1384 q^{6} + 2254 q^{7} + 3286 q^{8} + 4331 q^{9} + 5194 q^{10} + 5690 q^{11} + 5702 q^{12} + 5229 q^{13} + 4378 q^{14} + 3333 q^{15} + 2292 q^{16} + 1409 q^{17} + 763 q^{18} + 356 q^{19} + 138 q^{20} + 42 q^{21} + 9 q^{22} + q^{23}$ \\
$\b\a\a\b\a\b\a\a\b$ & $47143$ & $1 + 9 q + 42 q^{2} + 137 q^{3} + 351 q^{4} + 750 q^{5} + 1384 q^{6} + 2254 q^{7} + 3287 q^{8} + 4334 q^{9} + 5199 q^{10} + 5697 q^{11} + 5712 q^{12} + 5239 q^{13} + 4387 q^{14} + 3341 q^{15} + 2297 q^{16} + 1412 q^{17} + 764 q^{18} + 356 q^{19} + 138 q^{20} + 42 q^{21} + 9 q^{22} + q^{23}$ \\
$\b\a\b\a\a\b\a\a\b$ & $47159$ & $1 + 9 q + 42 q^{2} + 137 q^{3} + 351 q^{4} + 750 q^{5} + 1384 q^{6} + 2255 q^{7} + 3289 q^{8} + 4336 q^{9} + 5202 q^{10} + 5700 q^{11} + 5714 q^{12} + 5241 q^{13} + 4388 q^{14} + 3341 q^{15} + 2297 q^{16} + 1412 q^{17} + 764 q^{18} + 356 q^{19} + 138 q^{20} + 42 q^{21} + 9 q^{22} + q^{23}$ \\
$\b\a\b\a\a\b\a\b\a$ & $61111$ & $1 + 9 q + 43 q^{2} + 144 q^{3} + 378 q^{4} + 826 q^{5} + 1556 q^{6} + 2585 q^{7} + 3844 q^{8} + 5171 q^{9} + 6336 q^{10} + 7105 q^{11} + 7310 q^{12} + 6905 q^{13} + 5985 q^{14} + 4749 q^{15} + 3434 q^{16} + 2249 q^{17} + 1321 q^{18} + 687 q^{19} + 310 q^{20} + 118 q^{21} + 36 q^{22} + 8 q^{23} + q^{24}$ \\
    \caption{The values of $\mu(w)_{12}$ and $\mu_q(w)_{12}$ for the factors of the Fibonacci
    word of length up to 9 ordered in radix order. Factors of the Fibonacci word that are 
    Christoffel words and Markoff numbers are underlined.
    Christoffel words form a sparse subset of the language
    of the Fibonacci word.}
    \label{T:fibo_factor}
\end{longtable}

%%%%%%%%%%%%%%%%
% Bibliographie %
%%%%%%%%%%%%%%%%%
% \bibliographystyle{plain} %numeros
% \bibliographystyle{alpha} %author+year
\bibliographystyle{myalpha} %initials for first names
% \bibliographystyle{amsalpha} %initials for first names
%{\footnotesize
\bibliography{biblio}
%}

\end{document}